\pdfoutput=1
\documentclass[10pt,a4paper]{article}
\usepackage{amsmath,amsfonts,amssymb,amsthm,mathtools}
\usepackage{microtype}
\usepackage{url,hyperref}
\usepackage{indentfirst}
\hypersetup{colorlinks, linkcolor={red}, citecolor={blue}, urlcolor={black}}
\usepackage{comment}
\usepackage{hyperref}
\usepackage{cleveref}
\usepackage{enumerate,enumitem}
\usepackage{tikz}

\usepackage[margin=2.2cm]{geometry}

\usepackage[T1]{fontenc}
\usepackage{adjustbox}
\usepackage{subfig}
\usepackage{graphicx}
\usepackage{bm}
\usepackage{bbm}
\usepackage{esvect}
\usepackage{authblk}

\allowdisplaybreaks

\newtheorem{theorem}{Theorem}[section]
\newtheorem{lemma}[theorem]{Lemma}
\newtheorem{proposition}[theorem]{Proposition}

\newtheorem{claim}[theorem]{Claim}

\newtheorem{fact}[theorem]{Fact}
\newtheorem{problem}[theorem]{Problem}

\def\OPT{\mathrm{OPT}}
\def\supp{\mathrm{supp}}
\def\sc{\mathrm{sc}}

\newcommand{\multiset}[1]{\{\hspace{-0.25em}\{\hspace{0.1em}#1\hspace{0.1em}\}\hspace{-0.25em}\}}

\theoremstyle{definition}
\newtheorem{definition}[theorem]{Definition}
\begin{document}
\title{The inducibility of Tur\'an graphs}
\date{\today}
\author[1]{Xizhi Liu\thanks{Email: \texttt{liuxizhi@ustc.edu.cn}}}
\author[1,2]{Jie Ma\thanks{Email: \texttt{jiema@ustc.edu.cn}}}
\author[1]{Tianming Zhu\thanks{Email: \texttt{zhutianming@mail.ustc.edu.cn}}}
\affil[1]{{\small School of Mathematical Sciences, University of Science and Technology of China, Hefei, China}}
\affil[2]{{\small Yau Mathematical Sciences Center, Tsinghua University, Beijing, China}}

\maketitle

\begin{abstract}
    Let $I(F,n)$ denote the maximum number of induced copies of a graph $F$ in an $n$-vertex graph.
    The inducibility of $F$, defined as $i(F)=\lim_{n\to \infty} {I(F,n)}/{\tbinom{n}{v(F)}}$, is a central problem in extremal graph theory.
    In this work, we investigate the inducibility of Tur\'an graphs $F$. 
    This topic has been extensively studied in the literature, including works of Pippenger–Golumbic~\cite{Pippenger75}, Brown–Sidorenko~\cite{Brown94}, Bollob\'as--Egawa--Harris--Jin~\cite{Bollobas95}, Mubayi, Reiher, and the first author~\cite{LiuMubayi23}, and Yuster~\cite{Yuster26}. 
    Broadly speaking, these results resolve or asymptotically resolve the problem when the part sizes of $F$ are either sufficiently large or sufficiently small (at most four).
    
    We complete this picture by proving that for every Tur\'{a}n graph $F$ and sufficiently large $n$, the value $I(F,n)$ is attained uniquely by the $m$-partite Tur\'{a}n graph on $n$ vertices, where $m$ is given explicitly in terms of the number of parts and vertices of $F$.
    This confirms a conjecture of Bollob{\'a}s--Egawa--Harris--Jin~\cite{Bollobas95} from 1995, and we also establish the corresponding stability theorem.
    Moreover, we prove an asymptotic analogue for $I_{k+1}(F,n)$, the maximum number of induced copies of $F$ in an $n$-vertex $K_{k+1}$-free graph, thereby completely resolving a recent problem of Yuster~\cite{Yuster26}.   
    Finally, our results extend to a broader class of complete multipartite graphs in which the largest and smallest part sizes differ by at most on the order of the square root of the smallest part size.
\end{abstract}

\section{Introduction}\label{Sec:Introduction}
A fundamental problem in extremal graph theory is to determine the maximum number of induced copies of a given graph $F$ among all $n$-vertex graphs.
Formally, given two graphs $F$ and $G$, let $I(F,G)$ denote the number of induced copies of $F$ in $G$, that is, the number of subsets $S \subseteq V(G)$ of size $v(F)$ such that the induced subgraph $G[S]$ is isomorphic to $F$. Here, $v(F)$ denotes the number of vertices of $F$.
For every positive integer $n$, let 
\begin{align*}
    I(F, n) 
    \coloneqq \max\left\{ I(F, G) \colon v(G) = n \right\}. 
\end{align*}
The \emph{inducibility} of $F$ is then defined as
\begin{align*}
    i(F) 
    \coloneqq \lim_{n\to \infty} {I(F,n)}/{\tbinom{n}{v(F)}}.
\end{align*}
A systematic study of the inducibility problem for graphs was initiated in a foundational work of Pippenger--Golumbic~\cite{Pippenger75}, in which they established several general properties of inducibility and determined the inducibility of complete bipartite graphs with part sizes differing by at most one.

Determining $I(F,n)$ (or even $i(F)$) is rather difficult in general. For example, the inducibility of the path on four vertices remains wide open (see, e.g.,~\cite{Exo86,Even15note}).
For small graphs, bounds on the inducibility of graphs on $4$, $5$, and $6$ vertices have been obtained in several works, including Exoo~\cite{Exo86}, Hirst~\cite{Hir14}, Even-Zohar--Linial~\cite{Even15note}, Pikhurko--Slia{\v{c}}an--Tyros~\cite{Sliacan19}, and Bodn\'{a}r et al.~\cite{BGLLPS26}, some of which employ the computer-assisted flag algebra machinery of Razborov~\cite{Raz07}. 
A particularly interesting result obtained by Balogh--Hu--Lidick{\`y}--Pfender~\cite{Balogh16} is the determination of the inducibility of the $5$-cycle $C_{5}$, which confirms a special case of an old conjecture of Pippenger--Golumbic~\cite{Pippenger75} on the inducibility of cycles. 
This conjecture remains open for longer cycles, and improved upper bounds were obtained in~\cite{HT18,KNV19}.

For general graphs, results of Yuster~\cite{Yuster19} and Fox--Huang--Lee~\cite{FHL17} imply that for almost all graphs $F$,
\begin{align*}
    i(F) = \frac{v(F)!}{v(F)^{v(F)} - v(F)},
\end{align*}
where the lower bound construction arises from nested blowups of $F$ itself.
Answering a question of Bollob{\'a}s--Egawa--Harris--Jin~\cite{Bollobas95} asymptotically, Hatami--Hirst--Norine~\cite{Hatami14} established that if $F$ is a sufficiently large balanced blowup of some graph $K$, then the extremal graph for $I(F,n)$ is essentially a blowup of $K$.
Extending the results~\cite{MMNT19,KST19,FS20} on the edge-statistic conjecture of Alon--Hefetz--Krivelevich--Tyomkyn~\cite{AHKT20}, Ueltzen~\cite{Uel24} recently classified all graphs with high inducibility.

In this work, we focus on the case where $F$ is a complete multipartite graph, and more specifically, an almost balanced complete multipartite graph. 
This class of graphs has already been studied extensively since the work of Pippenger--Golumbic~\cite{Pippenger75}.
Throughout this work, for positive integers $a_1, \ldots, a_r$, let $K_{a_1, \ldots, a_r}$ denote the complete $r$-partite graph with part sizes $a_1, \ldots, a_r$. 
For positive integers $\ell \ge r$, the \emph{Tur\'{a}n graph} $T_r(\ell)$ denotes the complete $r$-partite graph on $\ell$ vertices in which the largest and smallest part sizes differ by at most one. 
The cases $\ell = r$ and $r=1$ correspond to the complete graph on $r$ vertices and the empty graph on $\ell$ vertices, respectively; both are trivial cases in the inducibility problem. So, we assume for the remainder of this work that $\ell \ge r+1$ and $r\ge 2$.

Results of Pippenger--Golumbic~\cite{Pippenger75} (also Bollob{\'a}s--Nara--Tachibana~\cite{BollobasNara86}) show that for bipartite Tur\'{a}n graphs $F$, the value of $I(F,n)$ is attained by bipartite Tur\'{a}n graphs $T_2(n)$.
Using Zykov symmetrization~\cite{Zyk49} together with additional arguments, Brown--Sidorenko~\cite{Brown94} showed that for every complete multipartite graph $F$, an extremal graph for $I(F,n)$ can always be chosen from the class of complete multipartite graphs; moreover, if $F$ is complete bipartite, then the number of parts can be taken to be at most two.
Nevertheless, determining the number of parts and the ratios of part sizes in the extremal construction for a complete multipartite graph $F$ remains difficult in general. Consequently, the inducibility of complete multipartite graphs is still largely open.

The family of Tur\'an graphs $T_r(\ell)$ has received substantial attention in the literature. 
Refining an asymptotic result of Brown--Sidorenko~\cite{Brown94} for the balanced complete $r$-partite graph $K_r(t):=T_r(rt)$, Bollob{\'a}s--Egawa--Harris--Jin~\cite{Bollobas95} proved that when $t \ge (1+o(1))\ln r$, the $r$-partite Tur{\'a}n graph $T_r(n)$ is the unique extremal graph, whereas this is not the case if $t < \frac{\ln(r+1)}{r\ln(1+1/r)}$.  
In the last section of~\cite{Bollobas95}, they remarked that (with $f(n, K_r(t))$ below corresponding to $I(K_r(t), n)$ in our notation):
\begin{quote}
    One may also venture the conjecture that for every pair $(r,t)$, $r \ge 4$, $t \ge 2$, if $n$ is sufficiently large then for some $s \ge 0$, $T_{r+s}(n)$ is the unique extremal graph for $f(n, K_r(t))$.
\end{quote}

This conjecture has been open for three decades.
There has been recent progress on the case when $F$ is a Tur\'{a}n graph with each part of small size.  
Mubayi, Reiher, and the first author~\cite[Theorem~1.13]{LiuMubayi23} determined the inducibility when $F$ is a Tur\'{a}n graph in which all but one part have size one.  
This result was extended very recently by Yuster~\cite[Theorem~1.6]{Yuster26}, who determined the inducibility when $F$ is an $r$-partite Tur\'{a}n graph with at most $3r+1$ vertices (in particular, all parts have size at most four).  
He also remarked that it would be highly interesting to determine whether the bound $3r+1$ can be removed, and posed a specific problem concerning the inducibility of Tur\'an graphs in clique-free graphs (see Problem~\ref{PROB:Yuster26} for details).

The main results of this work confirm the conjecture of Bollob{\'a}s--Egawa--Harris--Jin~\cite{Bollobas95} and resolves the above problem of Yuster~\cite{Yuster26} in a stronger form.

\subsection{Inducibility of almost balanced graphs}
Let $F = K_{a_1, \ldots, a_r}$ be the complete $r$-partite graph with part sizes $a_1, \ldots, a_r$, and assume that $a_1 \ge \cdots \ge a_r \ge 1$.
Let $\ell \coloneqq a_1 + \cdots + a_r$ denote the number of vertices of $F$.
We say that $F$ is \emph{almost balanced} if it is not the complete graph and 
\begin{align}\label{equ:def-almost-balance}
    \binom{a_1 - a_{r}}{2} < a_{r}, 
\end{align}
or, equivalently, $a_{1} < a_{r} + \tfrac{1}{2}\left(1+\sqrt{8 a_{r}+1}\right)$. 
It is key to observe that for $\ell\geq r+1$, every Tur\'an graph \(T_r(\ell)\) is almost balanced.
Define the constant (depending only on $F$)
\begin{align}\label{equ:def-kappaF}
    \kappa_F 
    \coloneqq \frac{\binom{\ell}{a_1, \ldots, a_r}}{\mathrm{sym}(a_1, \ldots, a_r)}
    = \frac{\ell!}{a_1! \cdots a_{r}! \cdot \mathrm{sym}(a_1, \ldots, a_r)}.
\end{align}
Here, $\mathrm{sym}(a_1, \ldots, a_r)$ denotes the size of the symmetry group of the multiset $\multiset{a_1, \ldots, a_{r}}$. 
In other words, if we let $b_1, \ldots, b_t$ denote the distinct elements of the multiple set $\multiset{a_1, \ldots, a_r}$, occurring with multiplicities $r_1, \ldots, r_t$, respectively,
then 
\begin{align*}
    \mathrm{sym}(a_1, \ldots, a_r)
    = r_1! \cdots r_{t}!. 
\end{align*}
Let $m_{r,\ell}$ denote the unique\footnote{The uniqueness will be proved in Lemma~\ref{Lemma:Unique-Maximizer}.} integer that maximizes the discrete function $f \colon [r, \infty) \to \mathbb{R}$ defined by 
\begin{align}\label{equ:def-fk}
    f(k) 
    \coloneqq \frac{(k-1)_{r-1}}{k^{\ell-1}}
    = \frac{(k-1)\cdots(k-r+1)}{k^{\ell-1}},
    \quad\text{for every integer $k \in [r,\infty)$}.
\end{align}

The main result of this subsection is the following theorem, which confirms the conjecture of Bollob{\'a}s--Egawa--Harris--Jin~\cite{Bollobas95} in a more general setting.

\begin{theorem}\label{THM:almost-balanced-exact}
    Suppose that $F$ is an almost balanced complete $r$-partite graph on $\ell \ge r+1$ vertices, and let $m = m_{r, \ell}$. Then
    \begin{align*}
        i(F) = \kappa_F \cdot \frac{(m-1)_{r-1}}{m^{\ell-1}}. 
    \end{align*}
    Moreover, there exists a constant $N_F$ such that for all $n \ge N_F$, the $m$-partite Tur\'{a}n graph $T_m(n)$ is the unique extremal graph for $I(F,n)$.  
\end{theorem}


In the balanced case $F = K_r(t)$ (i.e., every part of $F$ has the same size $t$), we are able to determine the exact value of $I(F,n)$ for \emph{every} positive integer $n$ (see Theorem~\ref{THM:Kr(t)-exact}). 

It is clear that $I(F,G) = I(\bar{F}, \bar{G})$, where $\bar{F}$ and $\bar{G}$ denote the complements of $F$ and $G$, respectively. Thus $I(F,n) = I(\bar{F}, n)$, and hence, Theorem~\ref{THM:almost-balanced-exact} also determines the inducibility of graphs that are vertex-disjoint unions of cliques, in which the smallest clique size $a_{r}$ and the largest clique size $a_{1}$ satisfy~\eqref{equ:def-almost-balance}.

\subsection{Inducibility of almost balanced graphs in $H$-free graphs}
Given a graph $H$, we say that a graph $G$ is {\it $H$-free} if it does not contain $H$ as a (not necessarily induced) subgraph.
Let $I_H(F, n)$ denote the maximum number of induced copies of $F$ in an $H$-free $n$-vertex graph, i.e., 
\begin{align*}
    I_{H}(F, n) 
    \coloneqq \max\big\{ I(F, G) \colon \text{$v(G) = n$ and $G$ is $H$-free} \big\}. 
\end{align*}
The \emph{$H$-free inducibility} of $F$ is then defined as
\begin{align*}
    i_{H}(F) 
    \coloneqq \lim_{n\to \infty} {I_{H}(F,n)}/{\tbinom{n}{v(F)}}.
\end{align*}
When $H = K_{k+1}$ is the complete graph on $k+1$ vertices, we simply write $I_{k+1}(F,n)$ and $i_{k+1}(F)$ instead of $I_{K_{k+1}}(F,n)$ and $i_{K_{k+1}}(F)$.

A classical example is the Erd\H{o}s Pentagon Problem~\cite{Erd84}, which asks for the determination of $I_{3}(C_5, n)$, that is, the maximum number of (induced) copies of $C_5$ in an $n$-vertex $K_{3}$-free graph.  
This problem was solved independently by Grzesik~\cite{Gre12} and Hatami--Hladk\'{y}--Kr\'{a}\v{l}--Norine--Razborov~\cite{HHKNR13} for large $n$, and subsequently by Lidick\'{y}--Pfender~\cite{LP18} for all $n$.

Very recently, Yuster~\cite{Yuster26} initiated a systematic study of inducibility problem in $H$-free graphs.
Among many other results, he determined (see~\cite[Theorem~1.6]{Yuster26}), for all $k \ge r$, the value of $i_{k+1}(F)$ when $F$ is an $r$-partite Tur\'{a}n graph on at most $3r+1$ vertices.  
He also remarked that it would be highly interesting to determine whether the analogous result holds for all $r$-partite Tur\'{a}n graphs, that is, without the restriction on the number of vertices.

\begin{problem}[{\cite[Problem~1.8]{Yuster26}}]\label{PROB:Yuster26}
    Is it true that for all $2 \le r < \ell$, there exists $t = t(r,\ell)$ such that the following holds?
    Let $F$ be the $r$-partite Tur\'{a}n graph on $\ell$ vertices.
    For all $k \le t$, $i_{k+1}(F)$ is attained asymptotically by the $k$-partite Tur\'{a}n graphs, and for all $k \ge t+1$, $i_{k}(F)$, and hence also $i(F)$, is attained asymptotically by the $t$-partite Tur\'{a}n graphs. 
\end{problem}

In the following theorem, we completely resolve Yuster’s problem (in fact, for a broader family) and determine the value of $t = t(r,\ell)$, namely the integer $m_{r,\ell}$ defined in the previous subsection.

\begin{theorem}\label{THM:K-free-almost-balanced-asymptotic}
    Suppose that $F$ is an almost balanced complete $r$-partite graph on $\ell \ge r+1$ vertices, and let $m = m_{r, \ell}$. Then 
    \begin{align*}
        i_{k+1}(F)
        = 
        \begin{cases}
            \kappa_F \cdot \frac{(k-1)_{r-1}}{k^{\ell-1}}, &\quad\text{if}\quad k \in [r, m-1], \\[.3em]
            \kappa_F \cdot \frac{(m-1)_{r-1}}{m^{\ell-1}}, &\quad\text{if}\quad k \in [m, \infty).
        \end{cases}
    \end{align*}  
\end{theorem}

\subsection{Perfect stability}\label{subsec:Perfect-Stability}
An interesting phenomenon, and also a very useful tool in extremal combinatorics, is \emph{stability}, introduced in the seminal work of Simonovits~\cite{Sim68}. 
Partly inspired by the work of Norin--Yepremyan~\cite{NY17tri}, a general framework for establishing a strong form of stability in certain graph extremal problems that can be solved using the Zykov symmetrization was developed recently by Liu--Pikhurko--Sharifzadeh--Staden~\cite{LiuPikhurko23}.
This framework will be used to prove the second  assertion (i.e. the exact result) of Theorem~\ref{THM:almost-balanced-exact}.

For two graphs $G$ and $H$ with the same number of vertices, the \emph{edit distance} $\mathrm{edit}(G, H)$ between $G$ and $H$ is the minimum number of edges one needs to add or remove from $G$ to make it isomorphic~to~$H$.

\begin{definition}[{\cite[Definition~2]{LiuPikhurko23}}]
    The inducibility problem for a complete multipartite graph $F$ is \emph{perfectly stable} if there exists a constant $C > 0$ such that for every $n$-vertex graph $G$ with $n \geq C$, there is an $n$-vertex complete multipartite graph $H$ satisfying
    \begin{align*}
        \mathrm{edit}(G,H) 
        \leq C  \cdot \frac{I(F,n)-I(F,G)}{\tbinom{n}{v(F)}} \binom{n}{2}.
    \end{align*}
\end{definition}
In particular, perfect stability of $F$ implies that for all sufficiently large $n$, \emph{every} extremal graph for the inducibility problem $I(F,n)$ is complete multipartite. 

By the result of Brown and Sidorenko~\cite[Proposition~1]{Brown94}, the inducibility problem for complete multipartite graphs can be solved using Zykov symmetrization. 
Thus, the framework from~\cite{LiuPikhurko23} can be applied to this problem. 
Indeed, as applications, Liu--Pikhurko--Sharifzadeh--Staden~\cite{LiuPikhurko23} established perfect stability for the inducibility problem of $F$ when $F$ is a complete bipartite graph, a complete $r$-partite graph with each part of size $t > 1 + \ln r$, and for small cases such as $K_{2,1,1,1}$ and $K_{3,1,1}$.  
They further conjectured~\cite[Conjecture~1]{LiuPikhurko23} that the inducibility problem is perfectly stable for every complete multipartite graph $F$. 

We confirm their conjecture for all almost balanced complete multipartite graphs.

\begin{theorem}\label{THM:almost-balanced-perfect-stability}
      Suppose that $F$ is an almost balanced complete multipartite graph. Then the inducibility problem $I(F,n)$ is perfectly stable.
\end{theorem}

\textbf{Organization}: 
In Section~\ref{Sec:Preliminaries}, we present some necessary definitions and preliminary results.  
In Section~\ref{SEC:asymptotic}, we determine the values of $i(F)$ and $i_{k+1}(F)$ for almost balanced complete multipartite graphs $F$.  
For technical reasons, the proofs are divided into two parts: first, we handle the case $\ell \le 2r-1$ (Subsection~\ref{SUBSEC:proof-asymptotic-small}), and then the case $\ell \ge 2r$ (Subsection~\ref{SUBSEC:proof-asymptotic-large}).  
In Section~\ref{Sec:Stablity}, we present the proof of Theorem~\ref{THM:almost-balanced-perfect-stability} (see Subsections~\ref{SUBSEC:proof-perfect-stability-small} and~\ref{SUBSEC:proof-perfect-stability-large}).  
In Section~\ref{Sec:Strongly-balanced-exact}, we determine the unique extremal construction for $I(F,n)$ for sufficiently large $n$, thereby completing the proof of Theorem~\ref{THM:almost-balanced-exact}.  
The last section (Section~\ref{SEC:remarks}) contains some concluding remarks.

\section{Preliminaries}\label{Sec:Preliminaries}
In this section, we introduce some necessary definitions and preliminary results.  

Let $k$ be a positive integer. 
For every real number $z$, define $\binom{z}{k} = (z)_{k}/k!$, where 
\begin{align*}
    (z)_{k}
    \coloneqq 
    \begin{cases}
        0, &\quad\text{if}\quad z \le k-1, \\
        z \cdots (z-k+1), &\quad\text{if}\quad z > k-1. 
    \end{cases}
\end{align*}
For a positive integer $k$, let $[k] \coloneqq \{1, \ldots, k\}$. 
We denote by $\mathbb{N}$ the set of natural numbers $\{0,1,2,\dots\}$, and by $\mathbb{N}_+$ the set of positive integers $\{1,2,3,\dots\}$.
Given a (possibly infinite) set $S \subseteq \mathbb{N}$, let $(S)_r$ denote the  collection of ordered $r$-tuples with pairwise distinct entries,  that is,
\begin{align*}
    (S)_r
    \coloneqq \big\{ (x_1, \ldots, x_r) \in S^{r} \colon \text{$x_i \neq x_j$ for all distinct $i,j \in [r]$} \big\}.
\end{align*}
In particular,
\begin{align*}
    (\mathbb{N})_r
    & \coloneqq \big\{ (x_1, \ldots, x_r) \in \mathbb{N}^{r} \colon \text{$x_i \neq x_j$ for all distinct $i,j \in [r]$} \big\}, \\[.2em]
    (\mathbb{N}_{+})_r
    & \coloneqq \big\{ (x_1, \ldots, x_r) \in \mathbb{N}_{+}^{r} \colon \text{$x_i \neq x_j$ for all distinct $i,j \in [r]$} \big\}.
\end{align*}

For a complete multipartite graph $F$, let $\sc(F)$ denote the number of parts of $F$ that have size exactly one, that is, the number of singleton parts.  
We have the following simple fact for counting the number of induced copies of $F$ in a complete multipartite graph.

\begin{fact}\label{FACT:Directly-Computation}
    Let $F = K_{a_1,\ldots,a_r}$ be the complete $r$-partite graph with part sizes $a_1 \ge \cdots \ge a_r \ge 1$.
    Suppose that $G$ is a complete multipartite graph whose non-singleton part sizes are $b_1,\ldots,b_k$.   
    Then 
    \begin{align*}
        I(F,G) 
        = \frac{1}{\mathrm{sym}(a_1, \ldots, a_{r})} \sum_{i = 0}^{\sc (F)} i! \binom{\sc (F)}{i} \binom{\sc (G)}{i} 
        \sum_{(i_1,\ldots, i_{r-i}) \in ([k])_{r-i}}~\prod_{j=1}^{r-i} \binom{b_{i_j}}{a_j}.
    \end{align*}
    In particular, if $\sc(F) = 0$ or $\sc(G) = 0$, then 
    \begin{align}\label{equ:IFG-formula}
        I(F,G) 
        = \frac{1}{\mathrm{sym}(a_1,\ldots,a_{r})} \sum_{(i_1,\ldots, i_{r}) \in ([k])_{r}}~\prod_{j=1}^{r} \binom{b_{i_j}}{a_j}.
    \end{align}
\end{fact}

A more convenient version of Fact~\ref{FACT:Directly-Computation} for the case in which each part of $G$ has approximately the same size and each part is large, which will be useful in the proof of perfect stability, is as follows.
\begin{fact}\label{FACT:Directly-Computation-handy}
    Let $F = K_{a_1,\ldots,a_r}$ be the complete $r$-partite graph with part sizes $a_1 \ge \cdots \ge a_r \ge 1$.  
    Suppose that $G$ is a complete $m$-partite graph in which each part has size $(1+o(1))n$, where $n$ is large.
    Then
    \begin{align}\label{equ:IFG-handy}
        I(F,G) 
        = (1+o(1)) \frac{(m)_{r}}{a_1! \cdots a_{r}! \cdot \mathrm{sym}(a_1,\ldots,a_{r})} n^{\ell}.
    \end{align}
\end{fact}

A foundation of our approach is the following result of Brown--Sidorenko~\cite{Brown94}.

\begin{lemma}[{\cite[Proposition 1]{Brown94}}]\label{Lemma:Sidorenko-Extremal-Partite}
    Suppose that $F$ is a complete multipartite graph with at least two parts. Then for every $n \ge 1$, there exists a complete multipartite graph $G$ on $n$ vertices such that $I(F,G) = I(F,n)$.
\end{lemma}

Roughly speaking, the main idea in the proof of Lemma~\ref{Lemma:Sidorenko-Extremal-Partite} is to show that Zykov symmetrization does not decrease the number of induced copies of $F$.  
Thus, after finitely many symmetrizations, one always ends up with a complete multipartite graph.
Since Zykov symmetrization also preserves $K_{k+1}$-freeness, the proof of Lemma~\ref{Lemma:Sidorenko-Extremal-Partite} implies that (as observed in~\cite{Yuster26}), when $F$ is complete multipartite, there exists an extremal construction for the inducibility problem $I_{k+1}(F, n)$ that is complete multipartite (with at most $k$ parts due to the $K_{k+1}$-freeness).

\begin{lemma}\label{Lemma:Sidorenko-Extremal-Partite-H-free}
    Suppose that $F$ is a complete $r$-partite graph with $r\ge 2$, and let $k \ge r$ be an integer. 
    Then for every $n \ge 1$, there exists a complete multipartite graph $G$ on $n$ vertices with at most $k$ parts such that $I(F,G) = I_{k+1}(F,n)$.
\end{lemma}

Therefore, by Lemmas~\ref{Lemma:Sidorenko-Extremal-Partite} and~\ref{Lemma:Sidorenko-Extremal-Partite-H-free}, in determining $i(F)$ and $i_{k+1}(F)$, it suffices to restrict our attention to the class of complete multipartite graphs, which, in the limit, reduces to an optimization problem of homogeneous polynomials (see~\cite[Corollary~6]{Brown94}).  
Since we also aim to prove perfect stability, we need to consider a slightly more general polynomial (introduced in~\cite{LiuPikhurko23}, see~\eqref{equ:pF-general}) and prove the finiteness of the set of maximizers.  
For this, we require the following definitions, mostly taken from~\cite[Section~2]{LiuPikhurko23}.

Define the \emph{partite limit space} by
\begin{align*}
    \overline{\mathcal{P}}
    \coloneqq \Big\{ (x_1,x_2,\ldots) \in \mathbb{R}^{\mathbb{N}_{+}} \colon \ x_1 \ge x_2 \ge \cdots \ge 0 \text{ and } \sum_{i \in \mathbb{N}_{+}} x_i \le 1 \Big\}.
\end{align*}

For every sequence $\bm{x} = (x_1, x_2, \ldots) \in \overline{\mathcal{P}}$, let $x_0 \coloneqq 1 - \sum_{i \in \mathbb{N}_{+}} x_i$.  
Define the \emph{support} of $\bm{x}$ as
\begin{align*}
    \supp(\bm{x}) 
    \coloneqq \{ i \in \mathbb{N}_{+} \colon x_i > 0 \}. 
\end{align*}

A sequence in $\overline{\mathcal{P}}$ gives rise to a complete multipartite graph on $n$ vertices via the following definition (the converse is straightforward).

\begin{definition}[{\cite[Definition~3]{LiuPikhurko23}}]\label{DEF:realization}
    Given $\bm{x} = (x_1, x_2, \ldots) \in \overline{\mathcal{P}}$ and $n \in \mathbb{N}$, the \emph{$n$-vertex realization} of $\bm{x}$, denoted $G_{n,\bm{x}}$, is the complete multipartite graph on the vertex set $[n]$ with parts $V_1, \ldots, V_m$ (for some $m$) and a set $V_0$ of universal vertices (i.e., the collection of singleton parts of $G_{n,\bm{x}}$), obtained as follows: if $x_0 = 0$, then take a partition $[n] = V_1 \cup \cdots \cup V_m$ such that $||V_i| - x_i n| < 1$ for each $i$, and set $V_0 = \varnothing$; otherwise, for all $i \ge 1$ with $x_i n \ge 2$, set $|V_i| = \lfloor x_i n \rfloor \ge 2$, and let $V_0$ contain the remaining vertices in $[n]$.  
    We call $V_0, \ldots, V_m$ the \emph{$\overline{\mathcal{P}}$-structure} of $G_{n,\bm{x}}$.
\end{definition}

For a tuple $(d_1, \ldots, d_t)$ of positive integers and a sequence $(x_1, x_2, \ldots) \in \mathbb{R}^{\mathbb{N}_{+}}$ of real numbers, let 
\begin{align}\label{equ:def-S-polynomial}
    S_{d_1, \ldots, d_t}(x_1, x_2, \ldots)
    \coloneqq \sum_{(i_1, \ldots, i_t) \in (\mathbb{N}_{+})_{t}} x_{i_1}^{d_1} \cdots x_{i_t}^{d_t}. 
\end{align}
For convenience, we set $S_{\varnothing}(x_1, x_2, \ldots) \coloneqq 1$. 

For two graphs $F$ and $G$, let $p(F, G) \coloneqq I(F, G) / \binom{v(G)}{v(F)}$ denote the \emph{induced density} of $F$ in $G$.

Let $a_1 \ge \cdots \ge a_r \ge 1$ be integers, and let $F = K_{a_1, \ldots, a_r}$.  
Define the \emph{limit function} 
\begin{align*}
    p_F(\bm{x}) 
    \coloneqq \lim_{n \to \infty} p(F, G_{n,\bm{x}}).
\end{align*}
It follows from Fact~\ref{FACT:Directly-Computation} that (recall the definition of $\kappa_{F}$ from~\eqref{equ:def-kappaF})
\begin{align}\label{equ:pF-general}
    p_F(\bm{x}) 
    = \kappa_F \cdot \sum_{i = 0}^{\sc(F)} \binom{\sc(F)}{i} x_0^i \cdot S_{a_1, \ldots, a_{r-i}}(\bm{x}).
\end{align}
In particular, if $\sc(F) = 0$ or $x_0=0$, this expression simplifies to 
\begin{align}\label{equ:pF-x0-zero}
    p_F(\bm{x}) 
    = \kappa_F \cdot S_{a_1, \ldots, a_r}(\bm{x}).
\end{align}

Intuitively, Lemma~\ref{Lemma:Sidorenko-Extremal-Partite} and Fact~\ref{FACT:Directly-Computation} together indicate that determining $i(F)$ for a complete multipartite graph amounts to finding the maximum of $p_F(\bm{x})$ over $\overline{\mathcal{P}}$.
A rigorous proof is given in~\cite[Section~2]{LiuPikhurko23}.

\begin{theorem}[\cite{LiuPikhurko23}]\label{THM:LiuPikhurko-Reduced-to-OPT}
    Suppose that $F$ is a complete multipartite graph. 
    Then 
    \begin{align*}
        i(F)
        = \max_{\bm{x}\in \overline{\mathcal{P}}}~p_{F}(\bm{x}).
    \end{align*}
\end{theorem}

For perfect stability, it is also necessary to understand the property of the maximizers of $p_F(\bm{x})$ over $\overline{\mathcal{P}}$.
Thus, for a complete multipartite graph $F$, we define 
\begin{align*}
    \OPT(F) 
    \coloneqq \left\{ \bm{x} \in \overline{\mathcal{P}} \colon p_F(\bm{x}) = i(F) \right\}.
\end{align*}
So for every $\bm{x} \in \OPT(F)$, the graph $G_{n,\bm{x}}$ is asymptotically extremal for the inducibility problem $I(F,n)$.

It follows relatively straightforwardly from Lemma~\ref{Lemma:Sidorenko-Extremal-Partite-H-free} and standard results from the theory of graphons (see~\cite[Section~1]{Yuster26} for further details) that, for a complete multipartite graph $F = K_{a_1,\ldots,a_r}$, we have
\begin{align}\label{equ:K-free-inducibility-reduction}
    i_{k+1}(F)
    = \max_{\bm{x} \in \mathbb{S}^{k-1}} p_F(\bm{x})
    = \kappa_F \cdot \max_{\bm{x} \in \mathbb{S}^{k-1}} S_{a_1, \ldots, a_{r}}(\bm{x}), 
\end{align}
where 
\begin{align*}
    \mathbb{S}^{k-1}
    \coloneqq \left\{(x_1, \ldots, x_{k}, 0, \ldots) \in \mathbb{R}^{\mathbb{N}_{+}} \colon \text{$x_i \ge 0$ for $i \in [k]$ and $x_1 + \cdots + x_{k} = 1$}\right\}. 
\end{align*}
Observe that $\mathbb{S}^{k-1} \subseteq \overline{\mathcal{P}}$ and the union $\bigcup_{k \ge 1} \mathbb{S}^{k-1}$ is the collection of sequences $\bm{x} \in \overline{\mathcal{P}}$ with $x_0 = 0$. 
Define 
\begin{align*}
    \OPT_{k+1}(F)
    \coloneqq \left\{ \bm{x} \in \mathbb{S}^{k-1} \colon p_F(\bm{x}) = i_{k+1}(F) \right\}.
\end{align*}

For convenience, set $\mathbb{S}^{\infty}=\bigcup_{k \ge 1} \mathbb{S}^{k-1}$, $\OPT_{\infty}(F) = \OPT(F)$, and $i_{\infty}(F) = i(F)$. 
For every $k \in \mathbb{N}_{+}$, define 
\begin{align*}
    \bm{k} \coloneqq \big(\underbrace{1/k,\, \ldots,\, 1/k}_{k},\, 0,\, \ldots\big).
\end{align*}
We call $\bm{x} = (x_1, x_2, \ldots) \in \overline{\mathcal{P}}$ \textit{balanced} if $x_i = x_j$ for all $i, j \in \supp(\bm{x})$; otherwise it is \textit{unbalanced}. 
In particular, if such a balanced $\bm{x}$ belongs to $\mathbb{S}^{k-1}$, then necessarily $\bm{x} = \bm{t}$ for some integer $t \le k$.

The following simple property of almost balanced complete multipartite graphs will be used multiple times. 
\begin{fact}\label{FACT:almost-balanced-sizes}
    Let $F = K_{a_1, \ldots, a_{r}}$ be an almost balanced complete $r$-partite graph with $a_1 \ge \cdots \ge a_{r}$, and let $\ell \coloneqq a_1 + \cdots + a_{r}$. 
    \begin{enumerate}[label=(\roman*)]
        \item\label{FACT:almost-balanced-sizes-1} If $\ell \in [r+1, 2r-1]$, then $a_{1} = \cdots = a_{\ell-r} = 2$ and $a_{\ell-r+1} = \cdots = a_{r} = 1$.
        \item\label{FACT:almost-balanced-sizes-2} If $\ell \ge 2r$, then $a_{r} \ge 2$. 
    \end{enumerate}
\end{fact}
\begin{proof}
    Suppose that $\ell \in [r+1, 2r-1]$. 
    By the Pigeonhole Principle, we have $a_r = 1$. 
    By the definition of almost balanced (see~\eqref{equ:def-almost-balance}), we have $\binom{a_1 - 1}{2} < 1$, which implies that $a_1 \le 2$. 
    Therefore, 
    \begin{align*}
        a_1 = \cdots = a_{\ell-r} = 2 
        \quad\text{and}\quad 
        a_{\ell-r+1} = \cdots = a_r = 1. 
    \end{align*}
    Now suppose that $\ell \ge 2r$. If $a_r = 1$, then the definition of almost balanced implies that $\binom{a_1 - 1}{2} < 1$, which forces $a_1 \le 2$. 
    Consequently, $\ell \le 2r-1$, which is a contradiction. 
\end{proof}

\section{Asymptotic results and uniqueness of the maximizer}\label{SEC:asymptotic}
In this section, we prove the following theorem, which constitutes the most crucial part of the proofs in this work.
In particular, it implies Theorem~\ref{THM:K-free-almost-balanced-asymptotic} and the asymptotic part of Theorem~\ref{THM:almost-balanced-exact}.

\begin{theorem}\label{THM:almost-balanced-asymptotic}
    Let $\ell > r \ge 2$ be integers, and let $m=m_{r,\ell}$. 
    Suppose that $F$ is an almost balanced complete $r$-partite graph on $\ell$ vertices. 
    Then 
    \begin{align*}
        \OPT_{k+1}(F)
        = 
        \begin{cases}
            \{ \bm{k} \}, & \quad\text{if}\quad k \in [r,~m-1], \\[.2em]
            \{ \bm{m} \}, & \quad\text{if}\quad k \in [m,~\infty]. 
        \end{cases}
    \end{align*}
    In particular, 
    \begin{align*}
        i_{k+1}(F)
        = 
        \begin{cases}
            \kappa_F \cdot \frac{(k-1)_{r-1}}{{k}^{\ell-1}}, & \quad\text{if}\quad k \in [r,~m-1], \\[.3em]
            \kappa_F \cdot \frac{(m-1)_{r-1}}{m^{\ell-1}}, & \quad\text{if}\quad k \in [m,~\infty]. 
        \end{cases}
    \end{align*}
\end{theorem}

In the next subsection, we establish some preliminary inequalities.
We then present the proof of Theorem~\ref{THM:almost-balanced-asymptotic} in two cases according to the size of $\ell$: $\ell \le 2r-1$ and $\ell \ge 2r$.

We remark that the proof of the case $\ell \le 2r-1$ in Theorem~\ref{THM:almost-balanced-asymptotic} is similar to that of~\cite[Theorem~1.6]{Yuster26} (which works for $\ell\leq 3r+1$), with one difference: for $k = \infty$, Yuster proved the uniqueness of the maximizer of~\eqref{equ:pF-x0-zero}, whereas we prove the uniqueness of the maximizer of~\eqref{equ:pF-general} (slightly stronger but necessary for perfect stability).
In contrast, our proof for the case $\ell \ge 2r$ differs substantially from Yuster’s. 
It relies on a quite nontrivial inequality (Proposition~\ref{PROP:a-b-ineq}), whereas Yuster’s argument is based on properties of low-degree polynomials.

\subsection{Basic properties of $f$ in~\eqref{equ:def-fk} and the integer $m_{r,\ell}$}
Let $h \colon I \to \mathbb{R}$ be a function on $I \subseteq \mathbb{R}$. 
We say that $h$ is \emph{unimodal} if there exists a point $x_0 \in I$ such that 
\begin{itemize}
    \item $h$ is non-decreasing on $I \cap (-\infty, x_0]$, and 
    \item $h$ is non-increasing on $I \cap (x_0, \infty)$. 
\end{itemize}

\begin{lemma}\label{Lemma:Unique-Maximizer}
    Let $\ell > r \ge 2$ be integers. The discrete function $f \colon [r,\infty) \to \mathbb{R}$, defined as in~\eqref{equ:def-fk}, is unimodal and attains its maximum at a unique integer $m_{r,\ell} \in [r,\infty)$.
\end{lemma}
\begin{proof}
    Consider the real extension $\tilde{f} \colon [r, \infty) \to \mathbb{R}$ of $f$ defined by 
    \[
        \tilde{f}(x) = \frac{(x-1)\cdots(x-r+1)}{x^{\ell-1}},  
        \quad\text{for every real number $x\in [r, \infty)$}.
    \]
    Straightforward calculations show that the derivative of $\tilde{f}(x)$ is 
    \[
        \tilde{f}'(x)
        = \frac{\tilde{f}(x)}{x}
          \left( \frac{x}{x-1} + \cdots + \frac{x}{x-r+1} - \ell + 1 \right).
    \]
    Define
    \begin{align}\label{equ:def-gx}
        g(x) \coloneqq \sum_{i=1}^{r-1} \frac{x}{x-i} - \ell + 1, 
        \quad\text{for every real number $x\in [r, \infty)$}.
    \end{align}
    Observe that $g(x)$ is strictly decreasing on $[r,\infty)$ and satisfies
    $\lim_{x\to \infty}g(x) = r - \ell < 0$. 
    Thus, the function $\tilde{f}$ is not monotone increasing and is unimodal.
    It follows that the discrete function $f$ can have at most two maximizers on the set of integers $\{r, r+1, \ldots\}$, and if there are two, they must be consecutive integers. 

    Suppose, for contradiction, that $f$ attains its maximum at two consecutive integers $k_0$ and $k_0+1$. 
    Note that, for every $k \in \{r, r+1, \ldots\}$, we have $f(k) < 1$, in particular, $f(k)$ is not an integer.
    Moreover, when written in lowest terms, the denominator of $f(k_0)$ divides $k_0^{\ell-1}$, while the denominator of $f(k_0+1)$ divides $(k_0+1)^{\ell-1}$. Since $k_0$ and $k_0+1$ are coprime, $k_0^{\ell-1}$ and $(k_0+1)^{\ell-1}$ are also coprime, and consequently the denominators of $f(k_0)$ and $f(k_0+1)$ are coprime (and both greater than $1$, since both $f(k_0)$ and $f(k_0+1)$ are not integers). Hence, $f(k_0)\neq f(k_0+1)$, a contradiction. Therefore, the maximizer of $f$ on $\{r,r+1,\ldots\}$ is unique.
\end{proof}

It follows from Lemma~\ref{Lemma:Unique-Maximizer} that $m_{r,\ell} = r$ if and only if $f(r) > f(r+1)$, which is equivalent to
\begin{align*}
    \ell > \frac{\ln(r+1)}{\ln(1+1/r)} = (1+o(1)) r \ln r.
\end{align*}
This is consistent with the result of Brown--Sidorenko~\cite[Theorem~9]{Brown94}. 
Refined bounds on $m_{r,\ell}$ will be provided in Lemma~\ref{Lemma:m-bounds}. 

\subsection{Proof of Theorem~\ref{THM:almost-balanced-asymptotic} for $\ell < 2r$}\label{SUBSEC:proof-asymptotic-small}
%
In this subsection we present the proof of Theorem~\ref{THM:almost-balanced-asymptotic} for $\ell < 2r$.

\begin{proof}[\bf Proof of Theorem~\ref{THM:almost-balanced-asymptotic} for $\ell < 2r$]
    Let $F = K_{a_1, \ldots, a_{r}}$ be an almost balanced complete $r$-partite graph on $\ell \in [r+1,~2r-1]$ vertices (i.e. $\ell = a_1 + \cdots + a_r$), and assume that $a_1 \ge \cdots \ge a_{r}$. 
    Let $m = m_{r,\ell}$ be the integer give by Lemma~\ref{Lemma:Unique-Maximizer}. 
    By Fact~\ref{FACT:almost-balanced-sizes}~\ref{FACT:almost-balanced-sizes-1}, we have $a_1 = \cdots = a_{\ell-r} = 2$ and $a_{\ell-r+1} = \cdots = a_r = 1$. 

    \textbf{Case 1}: $k \in [m, \infty]$. 

    \medskip 

    Since $\mathbb{S}^{k-1} \subseteq \overline{\mathcal{P}}$, it suffices to prove that $\OPT(F) = \{\bm{m}\}$ for this case. 
    
    Recall from~\eqref{equ:def-S-polynomial} and~\eqref{equ:pF-general} that for every $\bm{x} = (x_1, x_2, \ldots) \in \overline{\mathcal{P}}$, 
    \begin{align}\label{equ:express-pF-small}
        p_F(\bm {x})
        = \kappa_F \cdot \sum_{j = 0}^{2 r-\ell} \tbinom{2 r-\ell}{j} x_0^j  \, \sum_{(i_1, \ldots, i_{r-j}) \in (\mathbb{N}_{+})_{r-j}} x_{i_1}^{a_1} \cdots x_{i_{r-j}}^{a_{r-j}}.
    \end{align}
    
    \begin{claim}\label{Claim:OPT-all-equal-b<=2}
        The following statements hold. 
        \begin{enumerate}[label=(\roman*)]
            \item\label{Claim:OPT-all-equal-b<=2-i} If $\ell \ge r+2$, then every maximizer $\bm{x} \in \OPT(F)$ is balanced.
            \item\label{Claim:OPT-all-equal-b<=2-ii} If $\ell = r+1$, then for every unbalanced maximizer $\bm{x} \in \OPT(F)$, there exists a balanced maximizer $\bm{y} \in \OPT(F)$ such that
            \[
                y_0 = x_0 \quad \text{and} \quad |\supp(\bm{y})| < |\supp(\bm{x})|.
            \]
        \end{enumerate}
    \end{claim}
    \begin{proof}
        Suppose to the contrary that this claim is not true (for both items). 
        Let $\bm{x}=(x_1,x_2,\ldots) \in \OPT(F)$ be a counterexample with minimum size of support. 
        Then there exists some distinct pair $i, j \in \supp(\bm{x})$ such that $x_i\neq x_j$. 
        
        Fix all $x_k$ where $k \in \mathbb{N} \setminus \{ i,j \}$. Note that since $\sum_{k=1}^{\infty} x_k = 1$, the sum $\beta \coloneqq x_i + x_j$ is also fixed. From~\eqref{equ:express-pF-small}, there exist constants $K_1, \ldots, K_5$ independent of $x_i$ and $x_j$ such that 
        \begin{align}\label{equ:pF-general-b<=2}
            p_F(\bm{x})
            & = (x_i+x_j)K_1 +({x_i}^2+{x_j}^2)K_2+x_ix_jK_3+x_ix_j(x_i+x_j)K_4+(x_ix_j)^2K_5 \notag \\
            & = (x_i+x_j)K_1 +\left((x_i+x_j)^2 - 2x_ix_j \right)K_2+x_ix_jK_3+x_ix_j(x_i+x_j)K_4+(x_ix_j)^2K_5.
        \end{align}
        
        Since $a_1 = 2$ and $a_{r} = 1$, we have $\min\{K_1, K_2, K_4\} > 0$. 
        Moreover, $K_3 > 0$ if $a_{r-1} = 1$ (equivalently, $\ell\le 2r-2$), and $K_{5} > 0$ if $a_{2} = 2$ (equivalently, $\ell\ge r+2$). 

        Let $\alpha \coloneqq x_i x_j$. 
        Then $p_F(\bm{x})$ can be rewritten as 
        \begin{align*}
            \psi(\alpha)
            \coloneqq p_F(\bm{x})
            & = \beta K_1 + (\beta^2 - 2\alpha)K_2 + \alpha K_3 + \alpha\beta K_4 + \alpha^2 K_5 \\
            & = K_5 \alpha^2 + (\beta K_4 + K_3 - 2K_2) \alpha + (\beta K_1 + \beta^2 K_2). 
        \end{align*}
        Now observe that, while $\beta$ is fixed, the product $\alpha = x_i x_j$ can vary continuously from $0$ to $\beta^2/4$ by adjusting $x_i$ and $x_j$ under the constraint $x_i + x_j = \beta$. 
        Thus we may view $\psi(\alpha)$ as a quadratic polynomial in $\alpha$.
        
        Let $\bm{y}$ and $\bm{z}$ be the sequences in $\overline{\mathcal{P}}$ obtained from $\bm{x}$ by replacing $(x_i,x_j)$ with $(x_i+x_j,0)$ and $\left(\frac{x_i+x_j}{2}, \frac{x_i+x_j}{2}\right)$, respectively, and then reordering the entries accordingly. 
        Note that $y_0 = z_0 = x_0$. 
        
        Suppose that $\ell \ge r+2$. 
        Then $K_5 > 0$. 
        So $\psi(\alpha)$ is a quadratic polynomial in $\alpha$, and hence 
        \begin{align*}
            p_F(\bm{x})
            = \psi(\alpha)
            < \max\left\{\psi(0),~\psi\left(\tfrac{(x_i+x_j)^2}{4}\right)\right\}
            = \max\left\{p_F(\bm{y}),~p_F(\bm{z})\right\}. 
        \end{align*}
        However, this is a contradiction to the maximality of $p_F(\bm{x})$. 
        This proves~\ref{Claim:OPT-all-equal-b<=2-i}. 

        Now suppose that $\ell = r+1$. 
        Then $K_5 = 0$. Hence, $\psi(\alpha)$ is linear in $\alpha$ with coefficient $\beta K_4 + K_3 - 2K_2$. 
        If $\beta K_4 + K_3 - 2K_2 \neq 0$, then similarly, we have 
        \begin{align*}
            p_F(\bm{x})
            = \psi(\alpha)
            < \max\left\{\psi(0),~\psi\left(\tfrac{(x_i+x_j)^2}{4}\right)\right\}
            = \max\left\{p_F(\bm{y}),~p_F(\bm{z})\right\}, 
        \end{align*}
        again a contradiction. 
        So it must be the case that $\beta K_4 + K_3 - 2K_2 = 0$. 
        It follows that 
        \begin{align*}
            p_F(\bm{x})
            = \psi(\alpha)
            = \psi(0)
            = p_F(\bm{y}). 
        \end{align*}
        This means that $\bm{y} \in \OPT(F)$ as well. 
        Since $\bm{x}$ does not satisfy Claim~\ref{Claim:OPT-all-equal-b<=2}~\ref{Claim:OPT-all-equal-b<=2-ii}, the maximizer $\bm{y}$ must be unbalanced.
        Since $|\supp(\bm{y})| < |\supp(\bm{x})|$, it follows from the minimality of $\bm{x}$  that $\bm{y}$ satisfies Claim~\ref{Claim:OPT-all-equal-b<=2}~\ref{Claim:OPT-all-equal-b<=2-ii}.
        That is, there exists a balanced maximizer $\bm{w} = (w_1,w_2,\ldots) \in \OPT(F)$ with $w_0 = y_0 = x_0$ and $|\supp(\bm{w})| < |\supp(\bm{y})| < |\supp(\bm{x})|$.
        However, $\bm{w}$ witnesses that $\bm{x}$ itself satisfies the assertion of Claim~\ref{Claim:OPT-all-equal-b<=2}~\ref{Claim:OPT-all-equal-b<=2-ii}, contradicting the assumption that $\bm{x}$ is a counterexample. This completes the proof of Claim~\ref{Claim:OPT-all-equal-b<=2}.
    \end{proof}

    Now we establish the following inequality (which is slightly stronger than what we need) for sequences $\bm{x}\in \overline{\mathcal{P}}$ with $x_0=0$. 
    Recall that for every $t \in \mathbb{N}_{+}$, $\bm{t} = (1/t, \ldots, 1/t, 0, \ldots)$. 
    \begin{claim}\label{Claim:b<=2-similar-stability}
        There exists a constant $\varepsilon=\varepsilon_{r,\ell}>0$, such that for all $\bm{x}\in \overline{\mathcal{P}}$ with $x_0=0$, we have 
        \begin{align*}
            \Phi(\bm{x})
            \coloneqq p_F(\bm{x})+\varepsilon \bigg( \sum_{(i,j) \in (\mathbb{N}_{+})_{2}} x_ix_j -\tfrac{m-1}{m}  \bigg)^2 + \varepsilon \bigg( \sum_{(i,j,k) \in (\mathbb{N}_{+})_{3}} x_ix_jx_k-\tfrac{(m-1)(m-2)}{m^2} \bigg)^2 
            \le \Phi(\bm{m}), 
        \end{align*}
        and the equality holds if and only if $\bm{x} = \bm{m}$, where $m = m_{r, \ell}$ is the constant given by Lemma~\ref{Lemma:Unique-Maximizer}.
    \end{claim}
    \begin{proof}
        Recall that for any $t \in \mathbb{N}_{+}$, $p_F(\bm{t}) = \kappa_F \cdot \frac{(t-1)_{r-1}}{k^{\ell-1}} = \kappa_F \cdot f(t)$, where $f(t)$ is defined in~\eqref{equ:def-fk}.
        In what follows, we assume $f(r-1)=0$. 
        Define
        \begin{align*}
            \varepsilon \coloneqq \frac{\kappa_{F}}{3} \cdot \min\!\big\{ f(m) - f(m-1),\, f(m) - f(m+1) \big\}.
        \end{align*}
        It follows from Lemma~\ref{Lemma:Unique-Maximizer} that $\varepsilon > 0$. 

        Since $f(x)$ is unimodal (by Lemma~\ref{Lemma:Unique-Maximizer}), for every integer $t \neq m$, we have 
        \begin{align}\label{equ:Phi-max}
            \Phi(\bm{t})
            & = \kappa_{F} \cdot f(t) + \varepsilon \left(\tfrac{t-1}{t} - \tfrac{m-1}{m}\right)^2 + \varepsilon \left(\tfrac{(t-1)(t-2)}{t^2} - \tfrac{(m-1)(m-2)}{m^2}\right)^2 \notag \\[.3em]
            & \le \kappa_{F} \cdot \max\big\{ f(m-1),~f(m+1) \big\} + \varepsilon + \varepsilon \notag \\[.3em]
            & = \kappa_{F} \left( \max\big\{ f(m-1),~f(m+1) \big\} + \tfrac{2\kappa_F}{3} \cdot \min\big\{ f(m) - f(m-1),~f(m) - f(m+1) \big\} \right) \notag \\[.3em]
            & < \kappa_{F} \cdot f(m)
            = \Phi(\bm{m}). 
        \end{align}

        Suppose to the contrary that this claim fails. 
        Let $\bm{x} = (x_1,x_2,\ldots) \in \overline{\mathcal{P}} \setminus \{\bm{m}\}$ be a counterexample with $x_0 = 0$, i.e., $\Phi(\bm{x})\geq \Phi(\bm{m})$ but $\bm{x}\neq \bm{m}$.
        Using~\eqref{equ:Phi-max}, we conclude that there exist $i,j \in \supp(\bm{x})$ such that $x_i \neq x_j$. 
        Let $\alpha \coloneqq x_ix_j$ and $\beta \coloneqq x_i+x_j$.
        Similar to the proof of Claim~\ref{Claim:OPT-all-equal-b<=2} (more specifically,~\eqref{equ:pF-general-b<=2}), fix all $x_k$ where $k \in \mathbb{N} \setminus \{ i,j \}$. Note that since $\sum_{k=1}^{\infty} x_k = 1$, the sum $\beta = x_i + x_j$ is also fixed. So we may view $\Phi(\bm{x})$ as a polynomial in $\alpha$: 
        \begin{align*}
            \phi(\alpha)
            \coloneqq \Phi(\bm{x})
            = \ & p_F(\bm{x}) + \varepsilon \left( 2\alpha + J_{1} \right)^2 + \varepsilon \big( 6\alpha (1-\beta) + J_{2} \big)^2 \\
            = \ & \left(K_5 + 4\varepsilon + 36\varepsilon(1-\beta)^2 \right) \alpha^2 + \left(\beta K_4 + K_3 - 2 K_2 + 4\varepsilon J_{1} + 12\varepsilon(1-\beta)J_{2}\right) \alpha \\
            & + \left(\beta K_1 + \beta^2 K_2 + \varepsilon J_{1}^2 + \varepsilon J_{2}^2\right), 
        \end{align*}
        where $J_1, J_2$ are constants independent of $x_i, x_j$, and $K_1,\ldots , K_5$ are the same constants as in~\eqref{equ:pF-general-b<=2}.
        
        Let $\bm{y}$ and $\bm{z}$ be the sequences in $\overline{\mathcal{P}}$       
        obtained from $\bm{x}$ by replacing $(x_i,x_j)$ with $(x_i+x_j,0)$ and $\left(\frac{x_i+x_j}{2}, \frac{x_i+x_j}{2}\right)$, respectively, and then reordering the entries accordingly.
        Note that $y_0 = z_0 = x_0 = 0$. 
        Since $\phi(\alpha)$ is a quadratic polynomial in $\alpha$ with coefficient $K_5 + 4\varepsilon + 36\varepsilon(1-\beta)^2 > 0$, we have 
        \begin{align*}
            \Phi(\bm{x})
            = \phi(\alpha)
            < \max\left\{\phi(0),~\phi\left(\tfrac{(x_i+x_j)^2}{4}\right)\right\}
            = \max\left\{\Phi(\bm{y}),~\Phi(\bm{z})\right\},  
        \end{align*}
        a contradiction to the maximality of $\Phi(\bm{x})$. 
        This completes the proof of Claim~\ref{Claim:b<=2-similar-stability}. 
    \end{proof}

    \begin{claim}\label{CLAIM:OPT-x0-zero}
        Suppose that $\bm{x}\in \overline{\mathcal{P}}$ satisfies $x_0 > 0$. 
        Then $\bm{x} \not\in \OPT(F)$.
    \end{claim}
    \begin{proof}
        Suppose to the contrary that this is not true. 
        Then by Claim~\ref{Claim:OPT-all-equal-b<=2}, there exists a counterexample $\bm{x} \in \OPT(F)$ with $x_0>0$ and $x_1 = \cdots = x_{s} = y \coloneqq \frac{1-x_0}{s}$ for some $s \ge 1$. 
        For every integer $k \ge 1$, let $\bm{x}^{(k)}$ be the sequence obtained from $\bm{x}$ by appending the length-$k$ sequence $\frac{x_0}{k}, \ldots, \frac{x_0}{k}$, and then reordering the entries. 
        By applying \eqref{equ:express-pF-small} to $p_{F}(\bm{x})$ and \eqref{equ:pF-x0-zero} together with \eqref {equ:def-S-polynomial} to $p_{F}(\bm{x}^{(k)})$, we obtain that
        \begin{align*}
            p_{F}(\bm{x}) - p_{F}(\bm{x}^{(k)})
            & = \kappa_{F} \cdot \sum_{j=0}^{2r-\ell} \tbinom{2r-\ell}{j}(1-sy)^{j} (s)_{r-j} y^{\ell-j} \\
            & \qquad - \kappa_{F} \cdot \sum_{j=0}^{2r-\ell} \tbinom{2r-\ell}{j} (k)_{j} \left(\tfrac{1-sy}{k}\right)^{j} (s)_{r-j} y^{\ell-j} \\
            & \qquad - \kappa_{F} \cdot \sum_{i=1}^{\ell-r}\sum_{j=0}^{2r-\ell}\tbinom{\ell-r}{i}\tbinom{2r-\ell}{j} (k)_{i+j} \left(\tfrac{1-sy}{k}\right)^{2i+j} (s)_{r-i-j} y^{\ell-2i-j}. 
        \end{align*}
        Since $\frac{(k)_{j}}{k^j} \to 1$ and the third term goes to $0$ as $k \to \infty$ (because $\frac{(k)_{i+j}}{k^{2i+j}} \le \frac{1}{k} \to 0$), we have 
        \begin{align*}
            p_{F}(\bm{x}) - p_{F}(\bm{x}^{(k)})
            \to 0 \quad\text{as}\quad k \to \infty.
        \end{align*}
        %
        Combining this with Claim~\ref{Claim:b<=2-similar-stability}, we obtain 
        \begin{align*}
            0 \leq \Phi(\bm{x}^{(k)}) - p_{F}(\bm{x}^{(k)})
            & \le \Phi(\bm{m}) - p_{F}(\bm{x}^{(k)})\\
            & = p_{F}(\bm{m}) - p_{F}(\bm{x}^{(k)}) 
            \le p_{F}(\bm{x}) - p_{F}(\bm{x}^{(k)}) 
            \to 0 \quad\text{as}\quad k \to \infty.
        \end{align*}
        This implies that both
        \begin{align*}
            \lim_{k \to \infty} \sum_{(i,j) \in (\mathbb{N}_{+})_{2}} \bm{x}^{(k)}_i\bm{x}^{(k)}_j -\tfrac{m-1}{m}  = 0, \quad\text{and}\qquad
            \lim_{k \to \infty} \sum_{(i,j,k) \in (\mathbb{N}_{+})_{3}} \bm{x}^{(k)}_i\bm{x}^{(k)}_j\bm{x}^{(k)}_k-\tfrac{(m-1)(m-2)}{m^2} = 0.
        \end{align*}
        Simplifying this, we obtain 
        \begin{align*}
            \frac{m-1}{m} & = (1-sy)^2 + 2(1-sy)sy + s(s-1)y^2, \\
            \frac{(m-1)(m-2)}{m^2} & = (1-sy)^3 + 3(1-sy)^2sy + 3(1-sy)s(s-1)y^2 + s(s-1)(s-2) y^3. 
        \end{align*}
        It follows from the first equality that $y = \frac{1}{\sqrt{ms}}$. 
        Plugging it into the second equality we obtain 
        \begin{align*}
            1-\frac{3}{m}+\frac{2s}{(ms)^{3/2}} = \frac{(m-1)(m-2)}{m^2}, 
        \end{align*}
        which implies $s = m$. 
        However, this means that $x_0 = 1 - sy = 0$, contradicting the choice of $\bm{x}$.
    \end{proof}

    Suppose that $\bm{x} \in \OPT(F)$. 
    It then follows from Claim~\ref{CLAIM:OPT-x0-zero} that $x_0 = 0$, and from Claim~\ref{Claim:b<=2-similar-stability} that $\bm{x} = \bm{m}$.
    This shows that $\OPT(F) = \{\bm{m}\}$ and thereby completes the proof of Case 1. 

    \medskip

    \textbf{Case 2}: $k \in [r, m-1]$. 

    \medskip 

    An almost identical argument to the proof of Claim~\ref{Claim:OPT-all-equal-b<=2} yields the following result, so we omit the details here.
    
    \begin{claim}\label{CLAIM:OPT-all-equal-k-small}
        The following statements hold. 
        \begin{enumerate}[label=(\roman*)]
            \item\label{CLAIM:OPT-all-equal-k-small-1} If $\ell \ge r+2$, then every maximizer $\bm{x} \in \OPT_{k+1}(F)$ is balanced.
            \item\label{CLAIM:OPT-all-equal-k-small-2} If $\ell = r+1$, then for every unbalanced maximizer $\bm{x} \in \OPT_{k+1}(F)$, there exists a balanced  maximizer $\bm{y} \in \OPT_{k+1}(F)$ such that $|\supp(\bm{y})| < |\supp(\bm{x})|$. 
        \end{enumerate}
    \end{claim}

    Recall that for every integer $t \ge r$, we have $p_F(\bm{t}) = \kappa_F \cdot f(t)$, with $f(t)$ defined in~\eqref{equ:def-fk}. 
    By Lemma~\ref{Lemma:Unique-Maximizer}, the function $f$ is strictly increasing on $[r,k]$, so $p_F(\bm{t}) < p_F(\bm{k})$ for every $t < k$. Hence, $i_{k+1}(F) = p_F(\bm{k})$.

    If $\ell \ge r+2$, then Claim~\ref{CLAIM:OPT-all-equal-k-small}~\ref{CLAIM:OPT-all-equal-k-small-1} immediately implies that $\OPT_{k+1}(F) = {\bm{k}}$.
    
    Now, suppose $\ell = r+1$. Take an arbitrary maximizer $\bm{x} \in \OPT_{k+1}(F)$. If $\bm{x} \neq \bm{k}$, then by Claim~\ref{CLAIM:OPT-all-equal-k-small}~\ref{CLAIM:OPT-all-equal-k-small-2} we obtain another balanced maximizer $\bm{y} \in \OPT_{k+1}(F)$ with $|\supp(\bm{y})| < |\supp(\bm{x})|$. Consequently, $p_F(\bm{y}) = p_F(\bm{t})$ for some $t < k$, contradicting $i_{k+1}(F) = p_F(\bm{k}) > p_F(\bm{t})$. Therefore, we must have $\bm{x} = \bm{k}$, and again $\OPT_{k+1}(F) = \{\bm{k}\}$.
    This completes the proof of Theorem~\ref{THM:almost-balanced-asymptotic} for $\ell \le 2r-1$.
\end{proof}

\subsection{Proof of Theorem~\ref{THM:almost-balanced-asymptotic} for $\ell \ge 2r$}\label{SUBSEC:proof-asymptotic-large}
%
In this subsection, we present the proof of Theorem~\ref{THM:almost-balanced-asymptotic} for $\ell \ge 2r$.
We will use the following inequality, which constitutes the most crucial ingredient of our proof and distinguishes our approach from those in~\cite{Brown94,Bollobas95,LiuMubayi23,Yuster26}. 
\begin{proposition}\label{PROP:a-b-ineq}
    Let $x$ and $y$ be two positive real numbers. 
    For positive integers $d, d_1, d_2$ define
    \begin{alignat*}{2}
        \mu_{d} & \coloneqq (x + y)^{d} - x^d - y^d, \qquad
        & \omega_{d} & \coloneqq x^{d} + y^{d} - 2\left(\tfrac{x+y}{2}\right)^{d}, \\
        \psi_{d_1, d_2} & \coloneqq x^{d_1} y^{d_2} + x^{d_2} y^{d_1}, \qquad
        & \phi_{d_1,d_2} & \coloneqq 2\left(\tfrac{x+y}{2}\right)^{d_1+d_2} - x^{d_1} y^{d_2} - x^{d_2} y^{d_1}. 
    \end{alignat*}
    Suppose that $x \neq y$, and let $a,b,s,t$ be positive integers satisfying $b \ge t \ge s \ge a >\binom{b-a}{2}$. 
    Then 
    \begin{equation}\label{eq:ratio-inequality}
        \frac{\omega_{t}}{\mu_{t}}
        \le \frac{\omega_{a}}{\mu_{a}}
        < \frac{\phi_{s,t}}{\psi_{s,t}}. 
    \end{equation}
    Equivalently, 
    \begin{align}\label{eq:ratio-inequality-b}
        \mu_{t} \omega_{a} - \mu_{a} \omega_{t} \geq 0, 
        \quad\text{and}\qquad  
        \mu_{a} \phi_{s,t} - \omega_{a} \psi_{s, t} > 0.
    \end{align}
\end{proposition}

\begin{proof}
    First we prove that $\frac{\omega_{t}}{\mu_{t}} \le \frac{\omega_{a}}{\mu_{a}}$, which follows from the following claim. 

    \begin{claim}\label{CLAIM:key-inequality-I}
        We have $\frac{\omega_{k+1}}{\mu_{k+1}} \le \frac{\omega_{k}}{\mu_{k}}$ for every integer $k \ge a$. 
    \end{claim}
    \begin{proof}
        By definition, this is equivalent to showing 
        \[
            \left( x^{k+1} + y^{k+1} - 2\left(\tfrac{x + y}{2}\right)^{k+1} \right) \left( (x + y)^{k} - x^{k} - y^{k} \right) 
            \leq \left( x^{k} + y^{k} - 2\left(\tfrac{x + y}{2}\right)^{k} \right) \left( (x + y)^{k+1} - x^{k+1} - y^{k+1} \right).
        \]
        Simplifying this reduces to showing that: 
        \begin{equation}\label{eq:rearranged-monotonicity}
            (x + y)^{k+1} \leq x^{k+1} + y^{k+1} + (2^k - 1)(x^k y + x y^k). 
        \end{equation}
        Since for every $i \in [k]$ it holds that  
        \begin{align*}
            (x^k y + x y^k) - (x^{i} y^{k+1-i} + x^{k+1-i} y^{i}) 
            = x y (x^{i-1} - y^{i-1})(x^{k-i} - y^{k-i}) 
            \ge 0, 
        \end{align*}
        it follows that 
        \begin{align*}
            \sum_{i=1}^{k} \tbinom{k+1}{i}(x^k y + x y^k) 
            & \ge \sum_{i=1}^{k} \tbinom{k+1}{i} x^{i} y^{k+1-i} + \sum_{i=1}^{k} \tbinom{k+1}{i} x^{k+1-i} y^{i}, 
        \end{align*}
        which implies that 
        \begin{align*}
            \left(2^{k+1}-2\right)(x^k y + x y^k) 
            \ge 2\left( (x+y)^{k+1} - \left(x^{k+1} + y^{k+1}\right) \right)
        \end{align*}
        After dividing both sides by $2$ and rearranging, we obtain~\eqref{eq:rearranged-monotonicity}.  
    \end{proof}

    Next, we prove that $\frac{\omega_{a}}{\mu_{a}} < \frac{\phi_{s,t}}{\psi_{s,t}}$ for $b\geq t\geq s\geq a\geq 2$.
    By definition, this inequality is equivalent to
    \begin{align*}
        \big( (x + y)^{a} - x^{a} - y^{a} \big) \left( 2\left(\tfrac{x + y}{2}\right)^{s+t} - x^{s} y^{t} - x^{t} y^{s} \right) 
        > \left( x^{a} + y^{a} - 2\left( \tfrac{x + y}{2} \right)^{a} \right) \big( x^s y^t + x^t y^s \big).
    \end{align*}
    Simplifying this reduces to showing that: 
    \begin{align}\label{eq:simplification-RHS}
        \left( \tfrac{x+y}{2} \right)^{s+t-a} \big( (x + y)^{a} - x^{a} - y^{a} \big) 
        > \left( 2^{a-1}-1 \right) \left( x^s y^t + x^t y^s \right).
    \end{align}
    
    We first prove the following auxiliary inequalities:
    \begin{claim}
        The following inequalities hold:
        \begin{align}
            2^{a-2} \big( (x + y)^{a} - x^{a} - y^{a} \big) 
            & \ge (2^a-2) x y (x+y)^{a-2}; \label{eq:reduction-to-Sidorenko} \\[.4em]
            2\left( \tfrac{x+y}{2} \right)^{s+t-2} 
            & > x^{s-1} y^{t-1} + x^{t-1} y^{s-1}.  \label{eq:Sidorenko}
        \end{align}
    \end{claim}
    \begin{proof}
        First, we prove~\eqref{eq:reduction-to-Sidorenko}. 
        It is straightforward to verify that~\eqref{eq:reduction-to-Sidorenko} holds with equality when $a \in \{2,3\}$, so we may assume that $a \ge 4$. 
        For $i \in [a-1]$, let 
        \begin{align*}
            M_i 
            \coloneqq 2^{a-2} \tbinom{a}{i} - (2^a-2)\tbinom{a-2}{i-1},  
        \end{align*}
        %
        noting that $M_{i} = M_{a-i}$.
        Let 
        \begin{align*}
            f(i)
            \coloneqq \tfrac{2^{a-2} \tbinom{a}{i}}{(2^a-2)\tbinom{a-2}{i-1}}
            = \tfrac{2^{a-2} a(a-1)}{(2^a-2)i(a-i)}.
        \end{align*}
        Observe that $f(i)$ is decreasing in $i$ for $i \in (0,a/2)$ and increasing in $i$ for $i \in (a/2, a)$. 
        Also 
        \begin{align*}
            f(1) = f(a-1) = \tfrac{2^{a-2} a}{2^a-2} > 1
            \quad\text{and}\quad 
            f(a/2) = \tfrac{2^{a-2} (a-1)}{(2^a-2)a} < 1. 
        \end{align*}
        So there exists $i_0 \in [2,a/2]$ such that 
        \begin{align*}
            \begin{cases}
                M_i > 0, &\quad\text{if}\quad i < i_0 \mbox{ or } i > a-i_0, \\
                M_i \le 0, &\quad\text{if}\quad i \in [i_0, a - i_0]. 
            \end{cases}
        \end{align*}
        For $i,j \in [1, a/2]$ with $i < j$, we have 
        \begin{align*}
            \left(x^{i} y^{a-i} + x^{a-i} y^{i}\right) - \left(x^{j} y^{a-j} + x^{a-j} y^j\right)
            = x^{i}y^{i} \left(x^{j-i} - y^{j-i} \right) \left(x^{a-i-j} - y^{a-i-j}\right)
            \ge 0. 
        \end{align*}

        Define $\gamma_{i} \coloneqq x^{i}y^{a-i}$ and $\beta_{i} \coloneqq (\gamma_{i} + \gamma_{a-i})/2$ for $i \in [1, a/2]$. 
        The inequality above shows that for $i < j \leq a/2$, we have $\gamma_i+\gamma_{a-i} \geq \gamma_j+ \gamma_{a-j}$ and consequently $\beta_i \geq \beta_j$.
        Also note that 
        \begin{align*}
            \sum_{i=1}^{a-1} M_{i}
            = 2^{a-2} \sum_{i=1}^{a-1} \tbinom{a}{i} - \left(2^{a}-2\right) \sum_{i=1}^{a-1} \tbinom{a-2}{i-1}
            = 2^{a-2} \left(2^a - 2\right) - \left(2^{a}-2\right) 2^{a-2}
            = 0, 
        \end{align*}
        Then, using the Binomial Theorem and the symmetry $M_i=M_{a-i}$ and $\beta_i=\beta_{a-i}$, we obtain that
        \begin{align*}
            & 2^{a-2} \big( (x + y)^{a} - x^{a} - y^{a} \big) 
            - (2^a-2) x y (x+y)^{a-2} \\
            = &\sum_{i=1}^{a-1} M_{i} \gamma_{i}
            = \frac{1}{2} \sum_{i=1}^{a-1} \left( M_{i} \gamma_{i} + M_{a-i} \gamma_{a-i} \right)
            = \sum_{i=1}^{a-1} M_{i} \beta_{i}
            = \sum_{i=1}^{i_0-1} M_{i} \beta_{i} + \sum_{i=i_0}^{a-i_0} M_{i} \beta_{i} +  \sum_{i=a-i_0+1}^{a-1} M_{i} \beta_{i} \\
            \geq &\left ( \sum_{i=1}^{i_0-1} M_{i} \right ) \min_{1\leq i \leq i_0-1} \beta_i +  \left ( \sum_{i=i_0}^{a-i_0} M_{i} \right ) \max_{i_0\leq i \leq a-i_0} \beta_i + \left ( \sum_{i=a-i_0+1}^{a-1} M_{i} \right) \min_{a-i_0+1\leq i \leq a-1} \beta_i \\
            =  &\left ( \sum_{i=1}^{i_0-1} M_{i} \right ) \beta_{i_0-1} +  \left ( \sum_{i=i_0}^{a-i_0} M_{i} \right ) \beta_{i_0} + \left ( \sum_{i=a-i_0+1}^{a-1} M_{i} \right) \beta_{a-i_0+1}
            = 2 \left( \sum_{i=1}^{i_0-1} M_{i} \right ) \left(\beta_{i_0-1} - \beta_{i_0} \right) 
            \geq 0, 
        \end{align*}
        where the last equation uses the symmetry and the fact that $\sum_{i=1}^{a-1} M_i=0$.
        This proves~\eqref{eq:reduction-to-Sidorenko}. 

        \bigskip 
        
        It remains to prove~\eqref{eq:Sidorenko}. 
        The argument is adapted from the proof of~\cite[Theorem 3]{Brown94}, with some modifications.
        Define, for every $z \in [0,1]$, 
        \begin{align*}
            g(z) \coloneqq z^{t-1}(1-z)^{s-1} + z^{s-1}(1-z)^{t-1}
        \end{align*}
        We first show that $g$ attains its maximum uniquely at $z = 1/2$.
        
        Letting $x \coloneqq (1-z)/z$.
        Straightforward calculations show that 
        \begin{align*}
            g'(z) = z^{t-1}(1-z)^{s-2} \Bigl( (s-1)x^{t-s+1} - (t-1)x^{t-s} + (t-1)x - s + 1 \Bigr),
        \end{align*}
        Define 
        \begin{align*}
            h(x) \coloneqq (s-1)x^{t-s+1} - (t-1)x^{t-s} + (t-1)x - s + 1
        \end{align*}
        Let $y \coloneqq x - 1$. Then for $x\ge 1$, we have 
        \begin{align*}
            h(x) &= (s-1)(1+y)^{t-s+1} - (t-1)(1+y)^{t-s} + (t-1)(1+y) - s + 1 \\
            &= \Bigl( (s-1)(t-s+1) - (t-1)(t-s) + (t-1) \Bigr) y 
            + \sum_{i=2}^{t-s+1} \bigl( (s-1)\tbinom{t-s+1}{i} - (t-1)\tbinom{t-s}{i} \bigr) y^i \\
            &= 2\Bigl( s - 1 - \tbinom{t-s}{2} \Bigr) y 
            + \sum_{i=2}^{t-s+1} \tbinom{t-s}{i-1} \frac{1}{i} 
            \Bigl( (s-1)(t-s+1) - (t-1)(t-s-i+1) \Bigr) y^i \\
            &\geq 2\Bigl( s - 1 - \tbinom{t-s}{2} \Bigr) y 
            + \sum_{i=2}^{t-s+1} \tbinom{t-s}{i-1} \frac{1}{i} 
            \Bigl( (s-1)(t-s+1) - (t-1)(t-s-1) \Bigr) y^i \\
            &= 2\Bigl( s - 1 - \tbinom{t-s}{2} \Bigr) y 
            + \sum_{i=2}^{t-s+1} \tbinom{t-s}{i-1} \frac{2}{i} \Bigl( s - 1 - \tbinom{t-s}{2} \Bigr) y^i 
            \ge 0, 
        \end{align*}
        where in the last inequality we used the assumption that $\binom{b-a}{2} \le a-1$ (and thus, $s - 1 \geq a - 1 \geq \binom{b-a}{2} \geq \binom{t-s}{2}$).
        
        It follows that $g'(z) \ge 0$ for $z\in(0,1/2)$, and hence, $g$ is nondecreasing on $[0, 1/2]$.
        Since $g$ is symmetric around $z=1/2$ and is a polynomial which is not linear (since $s, t \ge 2$), it attains its maximum on the interval $[0,1]$ uniquely at $z=1/2$. 
        Therefore,~\eqref{eq:Sidorenko} holds.
    \end{proof}
    
    Combining~\eqref{eq:reduction-to-Sidorenko} and~\eqref{eq:Sidorenko}, we obtain
    \begin{align*}
        \left( \tfrac{x+y}{2} \right)^{s+t-a} \big( (x + y)^{a} - x^{a} - y^{a} \big) 
        & \ge \left( \tfrac{x+y}{2} \right)^{s+t-a} \left(2^a-2\right) xy \left( \tfrac{x+y}{2} \right)^{a-2} \\[.4em]
        & > \tfrac{2^a-2}{2} xy \left(x^{s-1} y^{t-1} + x^{t-1} y^{s-1}\right) 
        = \left( 2^{a-1}-1 \right) \left( x^s y^t + x^t y^s \right), 
    \end{align*}
    which proves~\eqref{eq:simplification-RHS}. 
    This completes the proof of Proposition~\ref{PROP:a-b-ineq}.
\end{proof}

We are now ready to prove Theorem~\ref{THM:almost-balanced-asymptotic} for $\ell \ge 2r$. 
\begin{proof}[\bf Proof of Theorem~\ref{THM:almost-balanced-asymptotic} for $\ell \ge 2r$]
    Let $F = K_{a_1, \ldots, a_{r}}$ be an almost balanced complete $r$-partite graph on $\ell \ge 2r$ vertices, that is, $\ell = a_1 + \cdots + a_r$.
    Let $m = m_{r,\ell}$ be the integer give by Lemma~\ref{Lemma:Unique-Maximizer}. 
    Assume that $a_1 \ge \cdots \ge a_r$. 
    By Fact~\ref{FACT:almost-balanced-sizes}~\ref{FACT:almost-balanced-sizes-2}, we have $a_r \ge 2$. 
    Fix an integer $k \in [r, \infty]$.
    
    \begin{claim}\label{CLAIM:large-opt-x0}
        For every $\bm{x} \in \OPT_{k}(F)$, we have $x_0 = 0$. 
    \end{claim}
    \begin{proof}
        It follows from the definition of $\mathbb{S}^{k-1}$ that $x_0 = 0$ if $k \neq \infty$. Thus we may assume that $k = \infty$. 
        Suppose to the contrary that there exists $\bm{x} = (x_1, x_2, \ldots) \in \OPT(F)$ with $x_0 > 0$. 
        Since in every realization $G_{n,\bm{{x}}}$ (see Definition~\ref{DEF:realization}), no induced copy of $F$ can contain vertices from $V_0$ (because $a_1 \ge \cdots \ge a_r \ge 2$), we have  
        \begin{align*}
            p_{F}(\bm{x})
            = \kappa_{F} \cdot \sum_{(i_1, \ldots, i_{r}) \in (\mathbb{N}_{+})_{r}} x_{i_1}^{a_1} \cdots x_{i_{r}}^{a_r}. 
        \end{align*}
        Let $\tilde{\bm{x}}$ be the sequence obtained from $\bm{x}$ by replacing $x_1$ with $x_0+x_1$. 
        Note that $\tilde{\bm{x}} \in \overline{\mathcal{P}}$ and $p_{F}(\tilde{\bm{x}}) - p_{F}(\bm{x}) > 0$, a contradiction to the optimality of $\bm{x}$. 
    \end{proof}
    
    Fix an arbitrary optimal sequence $\bm{x} \in \OPT_{k}(F)$.
    By Claim~\ref{CLAIM:large-opt-x0}, we have $x_0 = 0$. 
    We aim to show that all nonzero entries of $\bm{x}$ are equal. 
    Suppose to the contrary that there exist $i, j \in \supp(\bm{x})$ with $x_i \neq x_j$. 
    
    Let $\bm{y}$ and $\bm{z}$ be the sequence in $\overline{\mathcal{P}}$ obtained from $\bm{x}$ by replacing $(x_i,x_j)$ with $(x_i+x_j,0)$ and $\left(\frac{x_i+x_j}{2}, 
    \frac{x_i+x_j}{2}\right)$, respectively, and then reordering the entries accordingly.
    Note that $y_0 = z_0 = x_0 = 0$ and $\{\bm{y}, \bm{z}\} \subseteq \mathbb{S}^{k-1}$.
    In addition, define $\tilde{\bm{x}}$ by replacing both $x_i$ and $x_j$ with $0$ and then reordering the entries.

    Let $\bm{a} \coloneqq (a_1, \ldots, a_r)$.
    Recall the definition of $S_{\bm{a}}(\bm{x})$ from \eqref{equ:def-S-polynomial}.
    We shall prove the following key inequality: 
    \begin{claim}\label{CLAIM:large-shifting-increase}
        We have 
        \begin{align}\label{equ:large-opt-key-inequality}
            S_{\bm{a}}(\bm{x}) 
            < \max \big\{ S_{\bm{a}}(\bm{y}),~S_{\bm{a}}(\bm{z})\}. 
        \end{align}
        Consequently, $p_{F}(\bm{x}) < \max \big\{ p_{F}(\bm{y}),~p_{F}(\bm{z})\}$.
    \end{claim}
    \begin{proof}
        For $p, q \in [r]$, let $\bm{a}^{p}$ be the $(r-1)$-tuple obtained from $\bm{a}$ by removing the element $a_{p}$; let $\bm{a}^{p,q}$ be the $(r-2)$-tuple obtained from $\bm{a}$ by removing the elements $a_{p}$ and $a_{q}$; 
        and let 
        \begin{align*}
            A_{p}
            \coloneqq S_{\bm{a}^{p}}(\tilde{\bm{x}})
            \quad\text{and}\quad 
            B_{p,q} 
            \coloneqq S_{\bm{a}^{p,q}}(\tilde{\bm{x}}).
        \end{align*}
        It is clear that $|\supp(\bm{x})| \ge r$. 
        Thus, $|\supp(\tilde{\bm{x}})| \ge r-2$, and hence, $B_{p,q} > 0$ for all $\{p, q\} \subseteq [r]$. 
        Define 
        \begin{alignat*}{2}
            \mu_{p} & \coloneqq (x_i + x_j)^{a_{p}} - x_i^{a_{p}} - x_j^{a_{p}}, \qquad
            & \omega_{p} & \coloneqq x_i^{a_{p}} + x_j^{a_{p}} - 2\left(\tfrac{x_i+x_j}{2}\right)^{a_{p}}, \\
            \psi_{p, q} & \coloneqq x_i^{a_p} x_j^{a_q} + x_i^{a_q} x_j^{a_p}, \qquad
            & \phi_{p, q} & \coloneqq 2\left(\tfrac{x_i+x_j}{2}\right)^{a_p + a_q} - x_i^{a_p} x_j^{a_q} - x_i^{a_q} x_j^{a_p}. 
        \end{alignat*}
        Suppose to the contrary that this claim fails. 
        Then we have 
        \begin{align}
            0 
            \leq S_{\bm{a}}(\bm{x}) - S_{\bm{a}}(\bm{y})
            & = -\sum_{p\in [r]} \mu_{p} A_p + \sum_{\{p,q\} \in \binom{[r]}{2}} \psi_{p,q} B_{p,q},  \label{eq:ineq1} \\[.3em]
            0 
            \leq S_{\bm{a}}(\bm{x}) - S_{\bm{a}}(\bm{z})
            & = \sum_{p\in [r]} \omega_{p} A_{p} - \sum_{\{p,q\} \in \binom{[r]}{2}} \phi_{p,q} B_{p,q}. \label{eq:ineq2}
        \end{align}
        Considering the linear combination $-\omega_{r} \times~\eqref{eq:ineq1} - \mu_{r} \times~\eqref{eq:ineq2}$, we obtain 
        \begin{align}\label{eq:point-of-contradiction}
            \sum_{p\in [r]} (\mu_{p} \omega_{r} - \mu_{r} \omega_{p}) A_p + \sum_{\{p,q\} \in \binom{[r]}{2}} (\mu_{r} \phi_{p,q} - \omega_{r} \psi_{p, q}) B_{p,q} \leq 0.
        \end{align}

        However, it follows from~\eqref{eq:ratio-inequality-b} and $B_{p,q} > 0$ that 
        \begin{align*}
            \sum_{p\in [r]} (\mu_{p} \omega_{r} - \mu_{r} \omega_{p}) A_p + \sum_{\{p,q\} \in \binom{[r]}{2}} (\mu_{r} \phi_{p,q} - \omega_{r} \psi_{p, q}) B_{p,q}
            > 0, 
        \end{align*}
        a contradiction to~\eqref{eq:point-of-contradiction}. 
    \end{proof}

    It follows from Claim~\ref{CLAIM:large-shifting-increase} and the optimality of $\bm{x}$ that $\bm{x}$ must be balanced. 
    Suppose that $|\supp(\bm{x})| = \hat{k}$ for some integer $\hat{k} \le k$. Then $\bm{x} = \bm{\hat{k}} = (1/\hat{k}, \ldots, 1/\hat{k}, 0, \ldots)$, and hence, $p_{F}(\bm{x}) = \kappa_{F} \cdot {(\hat{k})_{r}}/{\hat{k}^{r}}$. 
    It follows from Lemma~\ref{Lemma:Unique-Maximizer} and the maximality of $p_{F}(\bm{x})$ that $\hat{k} = m$ if $k \ge m$ and $\hat{k} = k$ if $k \le m-1$. 
    This completes the proof of Theorem~\ref{THM:almost-balanced-asymptotic} for $\ell \ge 2r$.
\end{proof}

\section{Perfect stability}\label{Sec:Stablity}
\subsection{Preparations}
In this subsection, we briefly review the sufficient conditions for perfect stability in the framework developed by~\cite{LiuPikhurko23}. Their framework applies to the broader class of “symmetrizable” functions, but here we focus specifically on the inducibility problem for complete multipartite graphs $F$. Very roughly, \cite[Theorem 1.1]{LiuPikhurko23} asserts that if $\OPT(F)$ is finite and each maximizer $\bm{x} \in \OPT(F)$ satisfies certain \emph{strictness} conditions, then the inducibility problem for $F$ is perfectly stable. 

For convenience, we will use the following simplified variant of~\cite[Theorem 1.1]{LiuPikhurko23} (see also~\cite[Theorem 8.2]{Basit25}).

For a graph $G$ and a pair $\{x,y\} \subseteq V(G)$ of vertices, let $G \oplus xy$ denote the graph with vertex set $V(G)$ and edge set $E(G) \bigtriangleup \{xy\}$.
That is, $G \oplus \{xy\}$ is obtained by adding the edge $\{x,y\}$ if it is not present in $G$, and by removing it otherwise.
Suppose that $G$ is a complete $t$-partite graph with parts $V_1,\ldots,V_t$.
For every $A \subseteq [t]$, let $G_A$ denote the graph obtained from $G$ by adding a new vertex that is adjacent precisely to all vertices in $\bigcup_{i \in A} V_i$.

\begin{theorem}\label{THM:Stability-Verification}
    Let $F$ be a complete multipartite graph with $\ell$ vertices. 
    Suppose that $\OPT(F)$ contains a unique sequence $\bm{x}$ with $x_0 = 0$ (i.e. $\sum_{i\in \mathbb{N}_{+}}x_i = 1$) and $t \coloneqq |\supp(\bm{x})| < \infty$.
    Suppose that there exist constants $\varepsilon > 0$ and $N_0$ such that for every $n \ge N_0$ the following statements hold for the $n$-vertex realization $G \coloneqq G_{n,\bm{x}}$ with parts $V_1, \ldots, V_t \colon$ 
    \begin{enumerate}[label=(S\arabic*), ref=S\arabic*]
        \item\label{cond:S1} For every pair $\{x,y\} \subseteq V(G)$, we have 
            \begin{align*}
                I(F, G) - I(F, G \oplus xy) 
                \geq \varepsilon n^{\ell-2}.
            \end{align*}
        \item\label{cond:S2} For every $A \subseteq [t]$ with $|A| \neq t-1$, we have 
            \[
               \min\big\{ I(F, G_{A^{\ast}}) \colon A^{\ast} \subseteq [t] \text{ and } |A^{\ast}| = t-1 \big\} - I(F, G_A) 
                \geq \varepsilon n^{\ell-1}.
            \]
    \end{enumerate}
    Then the inducibility problem for $F$ is perfectly stable.
\end{theorem}

We also need the following refined estimates for $m_{r,\ell}$, the integer given by Lemma~\ref{Lemma:Unique-Maximizer}. 
\begin{lemma}\label{Lemma:m-bounds}
    Let $\ell > r \ge 2$ be integers and let $m=m_{r, \ell}$. 
    Then 
    \begin{enumerate}[label=(\roman*)]
        \item\label{Lemma:m-lower bound} $m > \max \left \{\frac{\ell(r-1)}{2(\ell-r)},~r-1 \right\}$, and  
        \item\label{Lemma:m-upper-bound} $m < \min \left\{  \frac{\ell(\ell-1)}{2(\ell-r)},~\frac{r}{\alpha}  \right\}$, where $\alpha \coloneqq \alpha(\ell/r) \in ( 1- r^2/\ell^2 , 1)$ is the unique positive real root of the equation $\mathrm{e}^{(\ell/r)\cdot x}(1-x)=1$. 
        In particular, $m < \frac{r\ell^2}{\ell^2-r^2}$. 
    \end{enumerate}
\end{lemma}
\begin{proof}
    First, we prove~\ref{Lemma:m-lower bound}. 
    Since $m \ge r$ holds trivially, it suffices to show that $m > \frac{\ell(r-1)}{2(\ell-r)}$, equivalently, $\frac{1}{m+1} < \frac{2(\ell-r)}{r\ell+\ell-2r}$. 
    Note that $\frac{\ell(r-1)}{2(\ell-r)}$ is decreasing in $\ell > r$, and when $\ell=2r$, we have $\frac{\ell(r-1)}{2(\ell-r)} < r$. 
    Thus we may assume that $r<\ell < 2r$.

    By Lemma~\ref{Lemma:Unique-Maximizer}, the discrete function $f$ defined in~\eqref{equ:def-fk} is unimodal.
    Therefore, the optimality of $m$ implies that $f(m) > f(m+1)$, which yields
    \begin{align*}
        \left(1 - \frac{1}{m+1} \right)^{\ell} < 1 - \frac{r}{m+1}.
    \end{align*}
    
    Define $g(x) \coloneqq (1-x)^{\ell} + rx - 1$ for $x \in [0,1]$. 
    Then the inequality above implies that $g\left(\tfrac{1}{m+1}\right) < 0$. 
    
    Straightforward calculations show that the second derivative of $g$ is $g''(x) = \ell(\ell-1)(1-x)^{\ell-2}$, which is nonnegative for $x\in [0,1]$.
    Thus, the function $g$ is convex on the interval $[0,1]$. 
    Since $g(0) = 0$ and $g\left(\frac{1}{m+1}\right) < 0$, to show that $\frac{1}{m+1} < \frac{2(\ell-r)}{r\ell+\ell-2r}$, it suffices to prove the following claim. 

    \begin{claim}\label{Claim:m-lower-r-s}
        We have $g\left( \frac{2(\ell-r)}{r\ell+\ell -2r} \right) \ge 0$. 
    \end{claim}
    \begin{proof}
        It is equivalent to show that 
        \begin{align*}
            \left(\frac{(r-1)\ell}{(r+1)\ell -2r}\right)^{\ell} 
            = \left(1 - \frac{2(\ell-r)}{r\ell+\ell -2r}\right)^{\ell} 
            \geq 1 - \frac{2r(\ell-r)}{r\ell+\ell -2r}
            = \frac{(2r-\ell)(r-1)}{(r+1)\ell-2r}.
        \end{align*}
        Taking logarithms on both sides and then rearranging the terms, this simplifies to showing that
        \begin{align}\label{eq:m-lower bound}
            \ell\ln \ell + (\ell-1)\ln(r-1) - (\ell-1)\ln\big((r+1)\ell-2r\big) - \ln(2r-\ell) 
            \geq 0. 
        \end{align}

        Define 
        \begin{align*}
            \vartheta(x) 
            \coloneqq x\ln x + (x-1)\ln(r-1) - (x-1)\ln\big( (r+1)x-2r \big) - \ln(2r-x), \quad\text{for}\quad x \in [r,2r). 
        \end{align*}
        Straightforward calculations show that 
        \begin{align*}
            \vartheta(r) = \vartheta'(r) = 0
            \quad\text{and}\quad
            \vartheta''(x) 
            = \frac{4r^2 (x-r)\bigl((r+2)x-4r\bigr)}{\bigl((r+1)x-2r\bigr)^2(2r-x)^2x}.
        \end{align*}
        For $x \in [r, 2r)$ and $r \ge 2$, we have $(r+2)x - 4r \ge (r+2)r - 4r = r(r-2) \ge 0$. 
        Hence $\vartheta''(x) \geq 0$.
        It follows that $\vartheta(x) \geq 0$ for all $x \in [r,2r)$, which implies~\eqref{eq:m-lower bound}. This completes the proof of Claim~\ref{Claim:m-lower-r-s}.
    \end{proof}

    \medskip 
    
    It remains to prove~\ref{Lemma:m-upper-bound}. Let $t \coloneqq \ell / r$ and $h_t(x) \coloneqq \mathrm{e}^{tx}(1-x) - 1$. 
    First we show that the real number $\alpha$ defined in the lemma lies in the interval $(1-1/t^2, 1)$.
    Note that the derivative $h_t'(x) = \mathrm{e}^{tx}(t - 1 - tx)$ is positive when $x < 1 - 1/t$ and negative when $x > 1- 1/t$. 
    Thus, $h_t$ is unimodal on $(0, \infty)$. 
    Since $h_t(0) = 0$ and $h_t(1) = -1$, the equation $h_t(x) = 1$ has a unique positive root (i.e. $\alpha$), which lies in the interval $(0, 1)$. 

    To show that $\alpha >1-1/t^2$, it suffices to verify that the following quantity is positive:
    \begin{align*}
        h_t(1-1/t^2) = \frac{1}{t^2} \mathrm{e}^{t-\frac{1}{t}} - 1.
    \end{align*}
    Let $u(t) \coloneqq \mathrm{e}^{t-\frac{1}{t}}/t^2$. Since the derivative $u'(t)=u(t)(1-1/t)^2$ is positive for $t>1$, the function $u$ is strictly increasing on $(1,\infty)$. Consequently, $h_t(1-1/t^2) = u(t)-1 > u(1)-1=0$ for $t > 1$, as desired.
    This proves that $\alpha \in (1-1/t^2, 1)$. 
    
    Next, we prove the upper bound for $m$. 
    Since $\alpha < 1$ and straightforward calculations show that $\frac{\ell(\ell-1)}{2(\ell-r)} > r$ for all $\ell > r$, we may assume that $m > r$ (otherwise, we are done).

    By Lemma~\ref{Lemma:Unique-Maximizer}, the discrete function $f$ defined in~\eqref{equ:def-fk} is unimodal.
    Therefore, the optimality of $m$ implies that $f(m) > f(m-1)$, which yields 
    \begin{align}\label{eq:m-upper-bound}
        \frac{m-r}{m} < \left( 1- \frac{1}{m} \right)^{\ell}.
    \end{align}
    Combining it with the inequality 
    \begin{align*}
        \left( 1- \frac{1}{m} \right)^{\ell}\leq1-\frac{\ell}{m}+\binom{\ell}{2}\left(\frac{1}{m}\right)^2, 
    \end{align*}
    we obtain $m <  \frac{\ell(\ell-1)}{2(\ell-r)}$. 

    It remains to prove that $m < r/\alpha$. 
    Define $\theta \coloneqq r/m$. 
    Then it follows from~\eqref{eq:m-upper-bound} and the inequality $\ln(1+x) \le x$ for $x > -1$ that 
    \begin{align*}
        1-\theta 
        = \frac{m-r}{m}
        < \left( 1- \frac{1}{m} \right)^{\ell} 
        = \exp\left(\ell \ln\left(1-\frac{1}{m}\right) \right) 
        \le \exp\left(-\frac{\ell}{m}\right) 
        = \mathrm{e}^{-t\theta}.
    \end{align*}
    In other words, we have $h_t(\theta)< 0$.
    Since $h_t$ is unimodal on $(0,\infty)$ and $h_{t}(0) = 0$, we have $\theta > \alpha$, which implies that $m < r/\alpha$.
    This proves~\ref{Lemma:m-upper-bound}, and thus completing the proof of Lemma~\ref{Lemma:m-bounds}.
\end{proof}

\subsection{Proof of Theorem~\ref{THM:almost-balanced-perfect-stability} for $\ell < 2r$}\label{SUBSEC:proof-perfect-stability-small}
In this subsection, we present the proof of Theorem~\ref{THM:almost-balanced-perfect-stability} for $\ell < 2r$. 

\begin{lemma}\label{LEMMA:stable-small-inequality}
    Let $\ell, r$ be integers such that $r\ge 2$ and $r+1 \le \ell \le 2r-1$. 
    Let $m = m_{r,\ell}$ be the integer given by Lemma~\ref{Lemma:Unique-Maximizer}. 
    Then the discrete function $h \colon [0, m] \to \mathbb{Z}$ defined by 
    \begin{align*}
        h(q)
        \coloneqq (q)_{r-1}\cdot \bigl(2r-\ell+2(m-q)(\ell-r)\bigr), 
        \quad\text{for every}\quad q \in [0, m], 
    \end{align*}
    satisfies $h(q) < h(m-1)$ for all integers $q \in [0, m] \setminus \{m-1\}$. 
\end{lemma}
\begin{proof}
    Note that $h(m-1) = \ell (m-1)_{r-1} > 0$. So it suffices to prove that for every $q \in [r-1, m] \setminus\{m-1\}$, 
    \begin{align}\label{equ:small-stable-inequality}
        h(q)
        = (q)_{r-1}\cdot \bigl(2r - \ell + 2(m-q)(\ell - r)\bigr) 
        < \ell (m-1)_{r-1}.
    \end{align}
    For $q = m$, inequality~\eqref{equ:small-stable-inequality} becomes $(m)_{r-1}(2r-\ell) < \ell(m-1)_{r-1}$, which simplifies to $2(\ell-r)m > \ell(r-1)$. 
    This follows from the lower bound on $m$ given by Lemma~\ref{Lemma:m-bounds}~\ref{Lemma:m-lower bound}, and hence this case holds.
    It remains to show $h(q) < h(m-1)$ for every $q \in [r, m-2]$.

    \begin{claim}\label{CLAIM:stable-small-hq-der}
        Suppose that $q \in [r, m-1]$. Then $h(q) - h(q-1) > 0$. 
    \end{claim}
    \begin{proof}
        First, we consider the case $q = m-1$. 
        In this case we need to show $(m-2)_{r-1}(3\ell-2r) < \ell(m-1)_{r-1}$, which is equivalent to $2(\ell-r)m < 3r\ell-2r^2-\ell$. 
        Since $m < \frac{\ell(\ell-1)}{2(\ell-r)}$ (by Lemma~\ref{Lemma:m-bounds}~\ref{Lemma:m-upper-bound}), we have 
        \begin{align*}
            2(\ell-r)m - (3r\ell-2r^2-\ell)
            < 2(\ell-r) \cdot \frac{\ell(\ell-1)}{2(\ell-r)} - (3r\ell-2r^2-\ell)
            = (\ell-r)(\ell-2r) 
            < 0. 
        \end{align*}
        This proves that $h(m-2) < h(m-1)$. 
    
        Next, we consider the case $q \in [r, m-2]$. 
        It follows from $m < \frac{\ell(\ell-1)}{2(\ell-r)}$ (by Lemma~\ref{Lemma:m-bounds}~\ref{Lemma:m-upper-bound}) that 
        \begin{align*}
            2(\ell-r)(3r-m-1)+(r-1)(2r-\ell)
            & > 2(\ell-r)\left(3r-1-\tfrac{\ell(\ell-1)}{2(\ell-r)}\right)+(r-1)(2r-\ell) \notag \\
            & = 2(\ell-r)(3r-1) - \ell(\ell-1) + (r-1)(2r-\ell) \notag \\
            & 
            = (\ell-r)(4r-\ell) 
            >0.
        \end{align*}
        Since $q \le m-2$, we have 
        \begin{align*}
           (r-1)(m+1)-rq
            \ge (r-1)(m+1)-r(m-2)
            = 3r-1-m. 
        \end{align*}
        Consequently, 
        \begin{align}\label{equ:stable-small-inequality-1388}
            2(\ell-r)\bigl((r-1)(m+1)-rq\bigr)+(r-1)(2r-\ell) 
            \ge  2(\ell-r)(3r-m-1)+(r-1)(2r-\ell) 
            > 0. 
        \end{align}
        Now, fix $q \in [r, m-2]$, and let $\delta(q) \coloneqq h(q) - h(q-1)$. Then 
         \begin{align*}
            \delta(q)
            & = (q)_{r-1}\bigl(2r-\ell+2(m-q)(\ell-r)\bigr) - (q-1)_{r-1}\bigl(2r-\ell+2(m-q+1)(\ell-r)\bigr) \\
            & = (q-1)_{r-2} \left( q\bigl(2r-\ell+2(m-q)(\ell-r)\bigr) - (q-r+1)\bigl(2r-\ell+2(m-q+1)(\ell-r)\bigr) \right) \\
            & = (q-1)_{r-2} \left(2(\ell-r)\bigl((r-1)(m+1)-rq\bigr)+(r-1)(2r-\ell)\right) 
            > 0, 
        \end{align*}
        where the last inequality follows from~\eqref{equ:stable-small-inequality-1388}. 
        This proves Claim~\ref{CLAIM:stable-small-hq-der}. 
    \end{proof}
    
    It follows from Claim~\ref{CLAIM:stable-small-hq-der} that $h(q) < h(m-1)$ for every $q \in [r, m-2]$. 
    This proves Lemma~\ref{LEMMA:stable-small-inequality}. 
\end{proof}

\begin{proof}[\bf Proof of Theorem~\ref{THM:almost-balanced-perfect-stability} for $\ell < 2r$]
    Let $F = K_{a_1, \ldots, a_{r}}$ be an almost balanced complete $r$-partite graph on $\ell \in [r+1,~2r-1]$ vertices (i.e. $\ell = a_1 + \cdots + a_r$).
    Let $m = m_{r,\ell}$ be the integer given by Lemma~\ref{Lemma:Unique-Maximizer}. 
    Assume that $a_1 \ge \cdots \ge a_r$. 
    By Fact~\ref{FACT:almost-balanced-sizes}~\ref{FACT:almost-balanced-sizes-1}, we have $a_1 = \cdots = a_{\ell-r} = 2$ and $a_{\ell-r+1} = \cdots = a_r = 1$. 

    Let $\varepsilon > 0$ be a sufficiently small constant such that, in particular, 
    \begin{align*}
        \varepsilon
        \le \left(2 m^{\ell+1}{2^{\ell-r}(\ell-r)!(2r-\ell)!}\right)^{-1}. 
    \end{align*}
    %
    It follows from Theorem~\ref{THM:almost-balanced-asymptotic} that $\OPT(F) = \{ \bm{m} \}$.  
    Thus, by Theorem~\ref{THM:Stability-Verification}, it suffices to verify~\eqref{cond:S1} and~\eqref{cond:S2} for all sufficiently large $n$ (with $\varepsilon$ there corresponding to $\varepsilon$ here). 
    
    Fix a sufficiently large integer $n$ and let $k \coloneqq \lfloor n/m \rfloor$. For convenience, we will use $o(\cdot)$ to denote lower-order terms.
    Let $G \coloneqq G_{n,\bm{m}}$ be the $n$-vertex realization with parts $V_1, \ldots, V_m$, where $|V_i| = k$ for all $i \in [m]$.

    We first verify~\eqref{cond:S1} for all distinct $x,y \in V(G)$. 
    
    \begin{claim}\label{CLAIM:small-stable-S1-distinct}
        Suppose that $(x,y) \in V_i \times V_j$ for distinct $i,j \in [m]$. Then 
        \[
            I(F, G) - I(F, G \oplus xy) \geq \varepsilon n^{\ell-2}.
        \]
    \end{claim}
    \begin{proof}
        By symmetry, we may assume that $(x,y) \in V_1 \times V_2$. 
        Since $x$ and $y$ lie in different parts, the pair $\{x, y\}$ is an edge of $G$.
        Let $\tilde{G}$ denote the induced subgraph of $G$ on $V_3 \cup \cdots \cup V_{r}$, noting that $\tilde{G}$ is complete $(m-2)$-partite. 
        
        Let $\mathcal{C}$ denote the collection of all induced copies of $F$ in $G$ that contain $\{x,y\}$. 
        Let 
        \begin{align*}
            \mathcal{C}_{1,1}
            & \coloneqq \left\{T \in \mathcal{C} \colon \left(|V_1\cap V(T)|, |V_2\cap V(T)| \right) = (1,1) \right\},  \\[.2em]
            \mathcal{C}_{1,2}
            & \coloneqq \left\{T \in \mathcal{C} \colon \left(|V_1\cap V(T)|, |V_2\cap V(T)| \right) = (1,2) \right\}, \\[.2em]
            \mathcal{C}_{2,1}
            & \coloneqq \left\{T \in \mathcal{C} \colon \left(|V_1\cap V(T)|, |V_2\cap V(T)| \right) = (2,1) \right\}, \\[.2em]     
            \mathcal{C}_{2,2}
            & \coloneqq \left\{T \in \mathcal{C} \colon \left(|V_1\cap V(T)|, |V_2\cap V(T)| \right) = (2,2) \right\}, 
        \end{align*}
        noting that $\mathcal{C}_{1,1} \cup \mathcal{C}_{1,2} \cup \mathcal{C}_{2,1} \cup \mathcal{C}_{2,2} = \mathcal{C}$. 

        Observe that the size of $\mathcal{C}_{1,1}$ is $0$ if $\ell=2r-1$, and coincides with the value of $I\big( K_{a_1, \ldots, a_{r-2}}, \tilde{G} \big)$ if $\ell\le 2r-2$. 
        So, by Fact~\ref{FACT:Directly-Computation-handy}, we have 
        \begin{align*}
            |\mathcal{C}_{1,1}|
            & =
            \begin{cases}
                0, & \quad\text{if}\quad \ell = 2r-1 \\[.2em]
                 \frac{(1+o(1)) (m-2)_{r-2}}{2^{\ell-r} (\ell-r)! (2r-\ell-2)!} k^{\ell-2} & \quad\text{if}\quad \ell \le 2r-2
            \end{cases}
            \quad =  \frac{(1+o(1)) (m-2)_{r-2}}{2^{\ell-r} (\ell-r)! (2r-\ell)!} (2r-\ell)(2r-1-\ell) k^{\ell-2}. 
        \end{align*}

        Similar, we have 
        \begin{align*}
            |\mathcal{C}_{1,2}| + |\mathcal{C}_{1,2}|
            & = 2 \cdot I\big( K_{a_2, \ldots, a_{r-1}}, \tilde{G} \big)
            =  \frac{(1+o(1)) (m-2)_{r-2}}{2^{\ell-r-2} (\ell-r-1)! (2r-\ell-1)!} k^{\ell-2}, 
        \end{align*}
        and 
        \begin{align*}
            |\mathcal{C}_{2,2}|
            = 
            \begin{cases}
                0 & \text{if}\quad \ell = r+1 ~\\[.2em]
                 \frac{(1+o(1)) (m-2)_{r-2}}{2^{\ell-r-2} (\ell-r-2)! (2r-\ell)!} k^{\ell-2} & \text{if}\quad \ell \ge r+2 ~
            \end{cases} 
            \quad  =  \frac{(1+o(1)) (m-2)_{r-2}}{2^{\ell-r-2} (\ell-r)! (2r-\ell)!} (\ell-r)(\ell-r-1) k^{\ell-2}. 
        \end{align*}
        Therefore, 
        \begin{align}\label{equ:C-szie-ell-small}
            |\mathcal{C}|
            & = |\mathcal{C}_{1,1}| + |\mathcal{C}_{1,2}| + |\mathcal{C}_{1,2}| + |\mathcal{C}_{2,2}| \notag \\[.3em]
            & =  \frac{(1+o(1)) (m-2)_{r-2}}{2^{\ell-r} (\ell-r)! (2r-\ell)!} \left((2r-\ell)(2r-1-\ell) + 4(\ell-r)(2r-\ell) +4(\ell-r)(\ell-r-1)\right) k^{\ell-2} \notag \\[.3em]
            & = (1+o(1)) \frac{(m-2)_{r-2} (\ell^2 - 3\ell + 2r)}{2^{\ell-r} (\ell-r)! (2r-\ell)!} k^{\ell-2}. 
        \end{align}

        \medskip 

        Let $\mathcal{C}'$ denote the collection of all induced copies of $F$ in $G\oplus xy$ that contain $\{x,y\}$. 
        Observe that, in each copy in $\mathcal{C}'$, the pair $\{x,y\}$ must form a part of size two.
        Thus, the size of $\mathcal{C}'$ coincides with the value of $I\big( K_{a_2, \ldots, a_{r}}, \tilde{G} \big)$, and hence, by Fact~\ref{FACT:Directly-Computation-handy}, 
        \begin{align*}
            |\mathcal{C}'|
            = \frac{(1+o(1)) (m-2)_{r-1}}{2^{\ell-r-1} (\ell-r-1)! (2r-\ell)!} k^{\ell-2}. 
        \end{align*}
        Combining this with~\eqref{equ:C-szie-ell-small}, we obtain 
        \begin{align*}
            I(F,G)-I(F,G\oplus xy)
            = |\mathcal{C}| - |\mathcal{C}'|
            & =  \frac{(1+o(1)) (m-2)_{r-2}}{2^{\ell-r} (\ell-r)! (2r-\ell)!} \big(\ell^2 - 3\ell + 2r - 2(m-r)(\ell-r)\big) k^{\ell-2}. 
        \end{align*}
        By Lemma~\ref{Lemma:m-bounds}~\ref{Lemma:m-upper-bound}, we have $m < \frac{\ell(\ell-1)}{2(\ell-r)}$, which implies that 
        \begin{align*}
            \ell^2 - 3\ell + 2r - 2(m-r)(\ell-r) 
            & > \ell^2 - 3\ell + 2r - 2\left(\tfrac{\ell(\ell-1)}{2(\ell-r)} - r\right)(\ell-r) \\
            & = \ell^2 - 3\ell + 2r - \ell(\ell-1) + 2r(\ell-r) 
            = 2(r-1)(\ell - r). 
        \end{align*}
        Therefore, 
        \begin{align*}
            I(F,G)-I(F,G\oplus xy)
            \ge  (1+o(1)) \frac{2(r-1)(\ell - r) (m-2)_{r-2}}{2^{\ell-r} (\ell-r)! (2r-\ell)!}  k^{\ell-2}
            > \varepsilon n^{\ell-2}, 
        \end{align*}
        which proves Claim~\ref{CLAIM:small-stable-S1-distinct}. 
    \end{proof}

    \begin{claim}\label{CLAIM:small-stable-S1-same}
        Suppose that $\{x,y\} \subseteq V_i$ for some $i \in [m]$. Then 
        \[
            I(F, G) - I(F, G \oplus xy) \geq \varepsilon n^{\ell-2}.
        \]
    \end{claim}
    \begin{proof}
        By symmetry, we may assume that $\{x, y\} \subseteq V_1$. Note that the pair $\{x, y\}$ is not an edge of $G$.
        Let $G'$ denote the induced subgraph of $G$ on $V_2 \cup \cdots \cup V_{m}$. 

        It is clear that the number of induced copies of $F$ in $G$ that contain $\{x, y\}$ coincides with the value of $I(K_{a_2, \ldots, a_{r}}, G')$, which, by Fact~\ref{FACT:Directly-Computation-handy}, is 
        \begin{align*}
            (1+o(1)) \frac{(m-1)_{r-1}}{2^{\ell-r-1} (\ell-r-1)! (2r-\ell)!} k^{\ell-2}. 
        \end{align*}
        If $\ell = 2r-1$, then $F$ has only one part of size one, and hence, the number of induced copies of $F$ in $G \oplus xy$ that contain $\{x, y\}$ is $0$. 
        If $\ell \le 2r-2$, then the number of induced copies of $F$ in $G \oplus xy$ that contain $\{x, y\}$ coincides with the value of $I(K_{a_1, \ldots, a_{r-2}}, G')$, which, by Fact~\ref{FACT:Directly-Computation-handy}, is 
        \begin{align*}
            (1+o(1)) \frac{(m-1)_{r-2}}{2^{\ell-r} (\ell-r)! (2r-\ell-2)!} k^{\ell-2}. 
        \end{align*}
        Therefore, 
        \begin{align*}
            I(F,G)-I(F,G\oplus xy)
            & =  \frac{(1+o(1)) (m-1)_{r-2}}{2^{\ell-r} (\ell-r)! (2r-\ell)!} \big(2(\ell-r)(m-r+1) -(2r-\ell)(2r-1-\ell) \big) k^{\ell-2}. 
        \end{align*}
        By Lemma~\ref{Lemma:m-bounds}~\ref{Lemma:m-lower bound}, we have $m > \frac{\ell(r-1)}{2(\ell-r)}$. 
        Therefore, 
        \begin{multline*}
            2(\ell-r)(m-r+1) -(2r-\ell)(2r-1-\ell) \\
             > 2(\ell-r)\left(\tfrac{\ell(r-1)}{2(\ell-r)}-r+1\right) -(2r-\ell)(2r-1-\ell) 
             = (\ell-r)(2r-\ell). 
        \end{multline*}
        %
        It follows that 
        \begin{align*}
            I(F,G)-I(F,G\oplus xy)
            & \ge \frac{(1+o(1)) (m-1)_{r-2}}{2^{\ell-r} (\ell-r)! (2r-\ell)!}  (\ell-r)(2r-\ell)  k^{\ell-2} \\
            & = \frac{(1+o(1)) (m-1)_{r-2}}{2^{\ell-r} (\ell-r-1)! (2r-\ell-1)!}  k^{\ell-2}
            > \varepsilon n^{\ell-2}, 
        \end{align*}
        which proves Claim~\ref{CLAIM:small-stable-S1-same}.
    \end{proof}

    \medskip

    It follows from Claims~\ref{CLAIM:small-stable-S1-distinct} and~\ref{CLAIM:small-stable-S1-same} that~\eqref{cond:S1} holds. 
    So it remains to verify~\eqref{cond:S2}.
    
    \begin{claim}\label{CLAIM:small-stable-S2}
        Suppose that $B \subseteq [m]$ is a subset with $|B| = q \in [0,m]$. 
        Then 
        \begin{align*}
            |\tilde{\mathcal{C}}_{q}|
            & = \frac{(1+o(1)) h(q)}{2^{\ell-r} (\ell-r)! (2r-\ell)!} k^{\ell-1}, 
        \end{align*}
        where $h\colon [m] \to \mathbb{Z}$ is the discrete function defined in Lemma~\ref{LEMMA:stable-small-inequality}. 
    \end{claim}
    \begin{proof}
        Fix a subset $B \subseteq [m]$ with $|B| = q$. 
        Recall that $G_{B}$ is the graph obtained from $G$ by adding a new vertex $v_{\ast}$ that is adjacent to precisely all vertices in $\bigcup_{i \in B} V_i$. 
        By symmetry, we may assume that $B = \{1, \ldots, q\}$. 

        Let $\tilde{\mathcal{C}}_{q}$ denote the collection of all induced copies of $F$ in $G_A$ that contain $v_{\ast}$. 
        Let $\tilde{\mathcal{C}}_{q}^{1}$ denote the collection of all induced copies of $F$ in $G_A$ in which $v_{\ast}$ forms a part of size one, and let $\tilde{\mathcal{C}}_{q}^{2}$ denote the collection of all induced copies of $F$ in $G_A$ in which $v_{\ast}$ lies in a part of size two.

        Note that the size of $\tilde{\mathcal{C}}_{q}^{1}$ coincides with the value of $I(K_{a_1, \ldots, a_{r-1}}, K[V_1, \ldots, V_{q}])$, which, by Fact~\ref{FACT:Directly-Computation-handy}, is 
        \begin{align*}
             \frac{(1+o(1)) (q)_{r-1}}{2^{\ell-r} (\ell-r)! (2r-\ell-1)!} k^{\ell-1}. 
        \end{align*}

        Note that the size of $\tilde{\mathcal{C}}_{q}^{2}$ coincides with the value of $I(K_{a_2, \ldots, a_{r}}, K[V_1, \ldots, V_{q}) \cdot \sum_{i=q+1}^{m}|V_i|$, which, by Fact~\ref{FACT:Directly-Computation-handy},~is 
        \begin{align*}
            \frac{(1+o(1))  (q)_{r-1}}{2^{\ell-r-1} (\ell-r-1)! (2r-\ell)!} k^{\ell-2} \cdot (m-q) k
            = \frac{(1+o(1)) (m-q) (q)_{r-1}}{2^{\ell-r-1} (\ell-r-1)! (2r-\ell)!} k^{\ell-1}. 
        \end{align*}

        Therefore, 
        \begin{align*}
            |\tilde{\mathcal{C}}_{q}|
            = |\tilde{\mathcal{C}}_{q}^{1}| + |\tilde{\mathcal{C}}_{q}^{2}|
            & = \frac{(1+o(1)) (q)_{r-1}}{2^{\ell-r} (\ell-r)! (2r-\ell-1)!} k^{\ell-1} + \frac{(1+o(1)) (m-q) (q)_{r-1}}{2^{\ell-r-1} (\ell-r-1)! (2r-\ell)!} k^{\ell-1} \\
            & = \frac{(1+o(1)) (q)_{r-1}}{2^{\ell-r} (\ell-r)! (2r-\ell)!} \big((2r-\ell) + 2(m-q) (\ell-r-1)\big) k^{\ell-1} 
             = \frac{(1+o(1)) h(q)}{2^{\ell-r} (\ell-r)! (2r-\ell)!} k^{\ell-1}.  
        \end{align*}
        This proves Claim~\ref{CLAIM:small-stable-S2}. 
    \end{proof}

    Fix an arbitrary set $A \subseteq [m]$ with $|A| \neq m-1$. 
    %
    Let 
    \begin{align*}
        \Delta 
        \coloneqq \min\big\{ I(F, G_{A^{\ast}}) \colon A^{\ast} \subseteq [m] \text{ and } |A^{\ast}| = m-1 \big\} - I(F, G_A). 
    \end{align*}
    Then it follows from Claim~\ref{CLAIM:small-stable-S2} and Lemma~\ref{LEMMA:stable-small-inequality} that 
    \begin{align*}
        \Delta
        & = \frac{(1+o(1)) h(m-1)}{2^{\ell-r} (\ell-r)! (2r-\ell)!} k^{\ell-1} - \frac{(1+o(1)) h(|A|)}{2^{\ell-r} (\ell-r)! (2r-\ell)!} k^{\ell-1} \\
        & = \frac{h(m-1) - h(|A|)}{2^{\ell-r} (\ell-r)! (2r-\ell)!} k^{\ell-1} + o(k^{\ell-1}) 
        \ge \frac{1}{2^{\ell-r} (\ell-r)! (2r-\ell)!} k^{\ell-1} + o(k^{\ell-1}) 
        \ge \varepsilon n^{\ell-1}. 
    \end{align*}
    This shows that~\eqref{cond:S2} holds.   
    So it follows from Theorem~\ref{THM:Stability-Verification} that the inducibility problem for $F$ is perfectly stable, which completes the proof of Theorem~\ref{THM:almost-balanced-perfect-stability} for $\ell \le 2r-1$.
\end{proof}

\subsection{Proof of Theorem~\ref{THM:almost-balanced-perfect-stability} for $\ell \ge 2r$}\label{SUBSEC:proof-perfect-stability-large}
In this subsection, we prove Theorem~\ref{THM:almost-balanced-perfect-stability} for the case $\ell \ge 2r$.
The proof parallels that of the previous section, with several steps simplified by the fact that $a_r \ge 2$ (which follows from $\ell \ge 2r$).

\begin{lemma}\label{LEMMA:stable-large-inequality}
    Let $\ell, r$ be integers such that $r\ge 2$ and $\ell \ge 2r$. 
    Let $m = m_{r,\ell}$ be the integer given by Lemma~\ref{Lemma:Unique-Maximizer}. 
    Then the discrete function $H \colon [0, m] \to \mathbb{Z}$ defined by 
    \begin{align*}
        H(q)
        \coloneqq (m-q)(q)_{r-1}, 
        \quad\text{for every}\quad q \in [0, m], 
    \end{align*}
    satisfies $H(q) < H(m-1)$ for all integers $q \in [0, m] \setminus \{m-1\}$.
\end{lemma}
\begin{proof}
    The inequality $H(q) < H(m-1)$ is clear for $q = m$ and $q \le r-2$ since $H(q) = 0$ in these cases. 
    For the remaining case, it suffices to show that $H(q)$ is increasing in $q$ for $q \in [r-1, m-2]$. 

    It follows from Lemma~\ref{Lemma:m-bounds}~\ref{Lemma:m-upper-bound} that 
    \begin{align}\label{eq:m<4r/3}
        m 
        < \frac{r\ell^2}{\ell^2-r^2}
        \le \frac{r(2r)^2}{(2r)^2-r^2}
        = \frac{4r}{3}. 
    \end{align}
    
    Fix $q \in [r, m-1]$. Then we have 
    \begin{align*}
        H(q) - H(q-1)
        & = (m-q)(q)_{r-1} - (m-q+1)(q-1)_{r-1} \\
        & = (q-1)_{r-2} \big((m-q)q - (m-q+1)(q-r+1)\big) \\
        & = (q-1)_{r-2} \big((r-1)(m+1)-rq\big) \\
        & \ge (q-1)_{r-2} \big(3r-1-m\big)
        > (q-1)_{r-2} \left(\tfrac{5r}{3}-1\right)
        > 0. 
    \end{align*}
    Here, in the first inequality we used the assumption that $q \le m-2$, in the second inequality we used~\eqref{eq:m<4r/3}. 
    This completes the proof of Lemma~\ref{LEMMA:stable-large-inequality}. 
\end{proof}

\begin{proof}[\bf Proof of Theorem~\ref{THM:almost-balanced-perfect-stability}]
    Let $F = K_{a_1, \ldots, a_{r}}$ be an almost balanced complete $r$-partite graph on $\ell \geq 2r$ vertices, that is, $\ell = a_1 + \cdots + a_r$.
    Let $m = m_{r,\ell}$ be the integer given by Lemma~\ref{Lemma:Unique-Maximizer}. 
    Since $\ell \ge 2r$, it follows from Lemma~\ref{Lemma:m-bounds}~\ref{Lemma:m-upper-bound} that~\eqref{eq:m<4r/3} holds. 
    
    Assume that $a_1 \ge \cdots \ge a_r$. By Fact~\ref{FACT:almost-balanced-sizes}~\ref{FACT:almost-balanced-sizes-2}, we have $a_r \geq 2$.
    Let $b_1 > \cdots > b_t$ denote the distinct elements of the multiple set $\multiset{a_1, \ldots, a_r}$, occurring with multiplicities $r_1, \ldots, r_t$, respectively. 
    
    It follows from Theorem~\ref{THM:almost-balanced-asymptotic} that $\OPT(F) = \{ \bm{m} \}$. 
    Thus, by Theorem~\ref{THM:Stability-Verification}, it suffices to verify~\eqref{cond:S1} and~\eqref{cond:S2} for some constant $\varepsilon > 0$ and all sufficiently large $n$. 
    
    Fix a sufficiently large $n$ and let $k \coloneqq \lfloor n/m \rfloor$. For simplicity, we will use $o(\cdot)$ to denote lower-order terms.
    
    Let $G \coloneqq G_{n,\bm{m}}$ be the $n$-vertex realization with parts $V_1, \ldots, V_m$, where $|V_i| = k$ for all $i \in [m]$.
    Let 
    \begin{align*}
        \varepsilon 
        \coloneqq \Big(4 m^{\ell+1}\cdot\mathrm{sym}(a_1,\ldots,a_r)\cdot a_1! \cdots a_{r}!\Big)^{-1}. 
    \end{align*}
    We first verify~\eqref{cond:S1} in the following two claims. 
    
    \begin{claim}\label{CLAIM:large-stable-S1-same}
        Suppose that $\{x,y\} \subseteq V_{i}$ for some $i \in [m]$. Then 
        \[
            I(F, G) - I(F, G \oplus xy) \geq \varepsilon n^{\ell-2}.
        \]
    \end{claim}
    \begin{proof}
        By symmetry, we may assume that $\{x,y\} \subseteq V_{1}$. 
        Since each part of $F$ has size at least two, there is no induced copy of $F$ in $G \oplus xy$ containing $\{x,y\}$. 
        However, it is clear that the number of induced copies of $F$ in $G$ that contain $\{x,y\}$ and with exactly $a_{1}$ vertices in $V_1$ is at least 
        \begin{align*}
            \tbinom{|V_1|-2}{a_1-2} \cdot I(K_{a_2, \ldots, a_{r}}, K[V_2, \ldots, V_{m}]). 
        \end{align*}
        By Fact~\ref{FACT:Directly-Computation-handy}, this is 
        \begin{align*}
            (1+o(1)) \frac{1}{(a_1-2)!} k^{a_1-2} \frac{(m-1)_{r-1}}{a_2! \cdots a_{r!} \cdot \mathrm{sym}(a_2, \ldots, a_{r})} k^{a_2+ \cdots +a_{r}}
            > \varepsilon n^{\ell-2}. 
        \end{align*}
        It follows that $I(F, G) - I(F, G \oplus xy) \geq \varepsilon n^{\ell-2}$, which proves Claim~\ref{CLAIM:large-stable-S1-same}. 
    \end{proof}

    \begin{claim}\label{CLAIM:large-stable-S1-different}
        Suppose that $(x,y) \in V_i \times V_j$ for distinct $i,j \in [m]$. Then 
        \[
            I(F, G) - I(F, G \oplus xy) \geq \varepsilon n^{\ell-2}.
        \]
    \end{claim}
    \begin{proof}
        By symmetry, we may assume that $(x,y) \in V_1 \times V_2$. 
        Recall that $b_1 > \cdots > b_t$ are the distinct elements of the multiple set $\multiset{a_1, \ldots, a_r}$, occurring with multiplicities $r_1, \ldots, r_t$, respectively. 
        
        Let $\mathcal{C}$ denote the collection of all induced copies of $F$ in $G$ that contain $\{x,y\}$.
        For $\{i,j\} \subseteq [t]$, let $\mathcal{C}_{i,j}$ denote the collection of all induced copies of $F$ in $G$ that contain $\{x,y\}$ with $x$ lying in a part of size $b_i$ and $y$ lying in a part of size $b_j$.
        For each $i \in [t]$ with $r_i \ge 2$, let $\mathcal{C}_i$ denote the collection of all induced copies of $F$ in $G$ that contain $\{x,y\}$ with both $x$ and $y$ lying in parts of size $b_i$.

        For each $\{i,j\} \subseteq [t]$, let $F_{i,j}$ denote the $(r-2)$-partite subgraph of $F$ obtained by removing one part of size $b_i$ and one part of size $b_j$.
        Then, by Fact~\ref{FACT:Directly-Computation-handy}, we have
        \begin{align*}
            |\mathcal{C}_{i,j}|
            & = \tbinom{|V_1|-1}{b_i - 1} \tbinom{|V_2|-1}{b_j - 1} \cdot I(F_{i,j}, G[V_3, \ldots, V_{m}]) \\
            & = (1+o(1)) \frac{k^{b_i-1+b_j-1}}{(b_i-1)! (b_j-1)!} \frac{(m-2)_{r-2}}{(b_i! b_j!)^{-1} \prod_{x\in [t]} (b_x !)^{r_x} \cdot (r_ir_j)^{-1} \prod_{x\in [t]} r_x!} k^{\ell - b_i -b_j} \\
            & = (1+o(1))  \frac{(m-2)_{r-2}}{\prod_{x\in [t]} (b_x !)^{r_x} \cdot \prod_{x\in [t]} r_x!} r_i b_i r_j b_j k^{\ell - 2}. 
        \end{align*}
        For each $i \in [t]$ with $r_i \ge 2$, let $F_{i}$ denote the $(r-2)$-partite subgraph of $F$ obtained by removing two parts of size $b_i$.
        Then, by Fact~\ref{FACT:Directly-Computation-handy}, we have
        \begin{align*}
            |\mathcal{C}_{i}|
            & = \tbinom{|V_1|-1}{b_i - 1} \tbinom{|V_2|-1}{b_i - 1} \cdot I(F_{i}, G[V_3, \ldots, V_{m}]) \\
            & = (1+o(1)) \frac{k^{2b_i-2}}{((b_i-1)!)^2} \frac{(m-2)_{r-2}}{(b_i!)^{-2} \prod_{x\in [t]} (b_x !)^{r_x} \cdot (r_i (r_i-1))^{-1} \prod_{x\in [t]} r_x!} k^{\ell - 2b_i} \\
            & = (1+o(1)) \frac{(m-2)_{r-2}}{ \prod_{x\in [t]} (b_x !)^{r_x} \cdot \prod_{x\in [t]} r_x!} r_i (r_i-1) b_i^2  k^{\ell - 2b_i}. 
        \end{align*}
        Therefore, 
        \begin{align}\label{equ:large-C-size}
            |\mathcal{C}|
            & =  \frac{(1+o(1)) (m-2)_{r-2}}{\prod_{x\in [t]} (b_x !)^{r_x} \cdot \prod_{x\in [t]} r_x!}  k^{\ell - 2} \Big(\sum_{(i,j) \in ([t])_{2}} r_i b_i r_j b_j + \sum_{i \in [t], r_i \ge 2} r_i (r_i-1) b_i^2 \Big) \notag \\
            & = \frac{(1+o(1)) (m-2)_{r-2}}{a_1! \cdots a_r! \cdot \mathrm{sym}(a_1, \ldots, a_r)}  k^{\ell - 2} \sum_{(i,j) \in ([r])_{2}} a_i a_j. 
        \end{align}

        Let $\mathcal{C}'$ denote the collection of all induced copies of $F$ in $G \oplus xy$ that contain $\{x,y\}$. 
        Let $\omega$ denote the number of parts of size two in $F$. 
        Since $a_i \ge 2$ for every $i \in [r]$, we have (using~\eqref{eq:m<4r/3})
        \begin{align}\label{equ:large-stable-1750}
            \sum_{(i,j) \in ([r])_{2}} a_i a_j - 2(m-r)\omega
            > 4r(r-1) - 2 \left(\frac{4r}{3} - r\right) r
            = \frac{2}{3} r (5r-6). 
        \end{align}
        Suppose that $\omega = 0$.  
        Since every part of $F$ has size at least $3$ and $\{x,y\}$ in not and edge in $G\oplus xy$, there is no induced copy of $F$ in $G \oplus xy$, and hence, $|\mathcal{C}'| = 0$.
        
        Suppose that $\omega \ge 1$ (i.e. $a_r = 2$), then $b_t =2$ and $r_t = \omega$. 
        Note that in this case, the size of $\mathcal{C}'$ coincides with the value of $I(K_{a_1, \ldots, a_{r-1}}, K[V_3, \ldots, V_{m}])$, which, by Fact~\ref{FACT:Directly-Computation-handy}, is 
        \begin{align*}
             \frac{(1+o(1)) (m-2)_{r-1}}{a_1! \cdots a_{r-1}! \cdot \mathrm{sym}(a_1, \ldots, a_{r-1})} k^{\ell-2}
             = \frac{(1+o(1)) (m-2)_{r-2}}{a_1! \cdots a_{r}! \cdot \mathrm{sym}(a_1, \ldots, a_{r})} k^{\ell-2} 2 (m-r) r_t. 
        \end{align*}
        Combining this with~\eqref{equ:large-C-size} and~\eqref{equ:large-stable-1750}, we obtain 
        \begin{align*}
            I(F,G)-I(F,G\oplus xy)
            & = \frac{(1+o(1)) (m-2)_{r-2}}{a_1! \cdots a_{r}! \cdot \mathrm{sym}(a_1, \ldots, a_{r})} k^{\ell-2} \Big(\sum_{(i,j) \in ([r])_{2}} a_i a_j - 2(m-r) \omega \Big) \\
            & \ge \frac{(1+o(1)) 2r(5r-6) (m-2)_{r-2}}{3 a_1! \cdots a_{r}! \cdot \mathrm{sym}(a_1, \ldots, a_{r})} k^{\ell-2}
            \ge \varepsilon n^{\ell-2}. 
        \end{align*}
        This completes the proof of Claim~\ref{CLAIM:large-stable-S1-different}.
    \end{proof}

    By Claims~\ref{CLAIM:large-stable-S1-same} and~\ref{CLAIM:large-stable-S1-different},~\eqref{cond:S1} holds.  
    So it remains to verify~\eqref{cond:S2}.
    
    \begin{claim}\label{CLAIM:large-stable-S2}
        Suppose that $B \subseteq [m]$ is a set of size $q \in [0, m]$. 
        Then the number of induced copies of $F$ in $G_B$ that contain $v_{\ast}$ is
        \begin{align*}
            \frac{(1+o(1)) \ell \cdot H(q)}{\prod_{i\in [r]}a_i! \cdot \mathrm{sym}(a_1, \ldots, a_{r})} k^{\ell-1}, 
        \end{align*}
        where $H \colon [m] \to \mathbb{Z}$ is the discrete function defined in Lemma~\ref{LEMMA:stable-large-inequality}. 
    \end{claim}
    \begin{proof} 
        Fix a subset $B \subseteq [m]$ with $|B| = q$. 
        Recall that $G_{B}$ is the graph obtained from $G$ by adding a new vertex $v_{\ast}$ that is adjacent to precisely all vertices in $\bigcup_{i \in B} V_i$. 
        By symmetry, we may assume that $B = \{1, \ldots, q\}$. 

        Let $\tilde{\mathcal{C}}_{q}$ denote the collection of all induced copies of $F$ in $G_B$ that contain $v_{\ast}$. 
        Since $F$ is complete $r$-partite and $a_r \ge 2$, if $q \in [r-2] \cup \{m\}$, then $|\tilde{\mathcal{C}}_{q}| = 0$. 
        Thus, we may assume that $q \in [r-1, m-1]$. 

        Recall that $b_1 > \cdots > b_t$ are the distinct elements of the multiple set $\multiset{a_1, \ldots, a_r}$, occurring with multiplicities $r_1, \ldots, r_t$, respectively. 
        For each $i \in [t]$, let $F_{i}'$ denote the $(r-1)$-partite subgraph of $F$ by removing a part of size $b_i$, and let $\tilde{\mathcal{C}}_{q}^{i}$ denote the collection of all induced copies of $F$ in $G_B$ that contain $v_{\ast}$ such that $v_{\ast}$ lying in a part of size $b_i$. 
        Then the size $\tilde{\mathcal{C}}_{q}^{i}$ coincides with the value of $\sum_{j = q+1}^{m} \binom{|V_j|}{b_i-1} \cdot I(F_{i}', K[V_1, \ldots, V_q])$, which, by Fact~\ref{FACT:Directly-Computation-handy}, is 
        \[
            (1+o(1)) (m-q) \frac{k^{b_{i}-1}}{(b_i-1)!} \cdot \frac{(q)_{r-1}}{(b_i!)^{-1} \prod_{i\in [r]}a_i! \cdot r_i^{-1} \cdot \mathrm{sym}(a_1, \ldots, a_{r})} k^{\ell-b_i}
            =  \frac{(1+o(1)) (m-q) (q)_{r-1}}{\prod_{i\in [r]}a_i! \cdot \mathrm{sym}(a_1, \ldots, a_{r})} r_i b_i k^{\ell-1}. 
        \]
        It follows that 
        \begin{align*}
            |\tilde{\mathcal{C}}_{q}|
            & =  \frac{(1+o(1)) (m-q) (q)_{r-1}}{\prod_{i\in [r]}a_i! \cdot \mathrm{sym}(a_1, \ldots, a_{r})} k^{\ell-1} \cdot \sum_{i \in [t]}r_i b_i 
            = \frac{(1+o(1)) \ell (m-q) (q)_{r-1}}{\prod_{i\in [r]}a_i! \cdot \mathrm{sym}(a_1, \ldots, a_{r})} k^{\ell-1}. 
        \end{align*}
        This completes the proof of Claim~\ref{CLAIM:large-stable-S2}. 
    \end{proof}

    Fix $A \subseteq [m]$ with $|A| \neq m-1$, and let 
    \begin{align*}
        \Delta
        \coloneqq \min\big\{ I(F, G_{A^{\ast}}) \colon A^{\ast} \subseteq [m] \text{ and } |A^{\ast}| = m-1 \big\} - I(F, G_A). 
    \end{align*}
    It follows from Claim~\ref{CLAIM:large-stable-S2} and Lemma~\ref{LEMMA:stable-large-inequality} that 
    \begin{align*}
        \Delta
        & = \frac{(1+o(1)) \ell}{\prod_{i\in [r]}a_i! \cdot \mathrm{sym}(a_1, \ldots, a_{r})} k^{\ell-1} \big(H(m-1) - H(q)\big)
        \ge \frac{(1+o(1)) \ell}{\prod_{i\in [r]}a_i! \cdot \mathrm{sym}(a_1, \ldots, a_{r})} k^{\ell-1} 
        \ge \varepsilon n^{\ell-1}. 
    \end{align*}
    This proves \eqref{cond:S2}.     
    By Theorem~\ref{THM:Stability-Verification}, the inducibility problem for $F$ is perfectly stable, proving Theorem~\ref{THM:almost-balanced-perfect-stability}. 
\end{proof}

\section{Exact results and uniqueness of the extremal construction}\label{Sec:Strongly-balanced-exact}
In this section, we show that, for large $n$, the extremal graph for $I(F,n)$ is the Tur\'an graph if $F$ is an almost balanced complete multipartite graph, thus completing the proof of Theorem~\ref{THM:almost-balanced-exact}.
We begin with the following two lemmas.

\begin{lemma}\label{LEMMA:exact-Mean-Value}
    Let $t \ge s \ge 1$ be integers. Suppose $x, y, N$ are real numbers satisfying $y > x+1$ and $\max\{|x - N|,~|y - N|\} = o(N)$ as $N \to \infty$.
    Then
    \begin{align}
        \tbinom{x+1}{s} +\tbinom{y-1}{s} - \tbinom{x}{s} -\tbinom{y}{s}
        & = -\tfrac{s(s-1)}{s!} (y-x-1) N^{s - 2} + o(N^{s-2}); \label{equ:LEMMA:exact-Mean-Value-a} \\[.4em]
        \tbinom{x+1}{t}\tbinom{y-1}{s} + \tbinom{x+1}{s}\tbinom{y-1}{t} - \tbinom{x}{t}\tbinom{y}{s} - \tbinom{x}{s}\tbinom{y}{t}
        & = \tfrac{  t+s - (t-s)^2 }{s!t!} \big( y-x-1 \big) N^{t+s-2} + o(N^{t+s-2}).\label{equ:LEMMA:exact-Mean-Value-b} 
    \end{align}
\end{lemma}
\begin{proof}
    We start with the first equality. 
    The case $s=1$ is trivially true, so we may assume that $s \ge 2$. 
    Let $\phi \colon [s-1,\infty) \to \mathbb{R}$ be the function defined by $\phi(z) = \binom{z}{s-1}$ for all $z\in [s-1,\infty)$. 
    By the Mean Value Theorem, there exists $z_0 \in [x, y-1]$ (hence $|z_0-N|=o(N)$) such that
    \begin{align*}
        \tbinom{x+1}{s} + \tbinom{y-1}{s} - \tbinom{x}{s} - \tbinom{y}{s}
        = \tbinom{x}{s-1} - \tbinom{y-1}{s-1} 
        & = (x - y + 1) \phi'(z_0) \\
        & = (x-y+1) \left( \tfrac{z_0^{s-2}}{(s-2)!} +o\left(z_0^{s-2}\right) \right) 
        = -\tfrac{y-x-1}{(s-2)!} N^{s-2} +o(N^{s-2}), 
    \end{align*}
    as desired. 
    
    We now consider the second equality. 
    Let 
    \begin{align*}
        g(x,y) \coloneqq (x+1)(y-s)_{t-s+1} + (x+1)(y-s)(x-s+1)_{t-s} 
                      - (x-s+1)y(y-s)_{t-s} - y(x-s+1)_{t-s+1}.
    \end{align*}
    
    \begin{claim}\label{CLAIM:exact-Mean-Value-1928}
        We have $g(x,y) = \left(t+s - (t-s)^2\right) (y-x-1) N^{t-s} + o(N^{t-s})$. 
    \end{claim}
    \begin{proof}
        The case $t - s \in \{ 0,1 \}$ follows directly from the following fact (which itself follows from straightforward calculations):
        \begin{align*}
            g(x,y)
            = 
            \begin{cases}
                2s(y-x-1), & \quad\text{if}\quad t - s = 0, \\
                s(y-x-1)(x+y-2s), & \quad\text{if}\quad t - s = 1. 
            \end{cases}
        \end{align*}
        So we may assume that $t-s\ge 2$. 
        Simplifying $g(x,y)$, we obtain 
        \begin{align*}
            g(x,y)
            =s\bigl(y(y-s)_{t-s}-(x+1)(x+1-s)_{t-s}\bigr)
                  -t(x+1)y\Bigl(\tfrac{(y-s)_{t-s}}{y}-\tfrac{(x+1-s)_{t-s}}{x+1}\Bigr).
        \end{align*}
        Let $\psi_{1}, \psi_2 \colon [t-1, \infty) \to \mathbb{R}$ be functions defined by 
        \begin{align*}
            \psi_{1}(z) = z(z-s)_{t-s}
            \quad\text{and}\quad 
            \psi_{2}(z) = \frac{(z-s)_{t-s}}{z}
            \quad\text{for all}\quad z \in [t-1, \infty). 
        \end{align*}
        By the Mean Value Theorem, there exist $z_1,z_2 \in [x,y-1]$ such that
        \begin{align*}
            y(y-s)_{t-s}-(x+1)(x+1-s)_{t-s}
            & = (y-x-1) \psi_{1}'(z_1) \\
            & = (y-x-1) \left((t-s+1) z_1^{t-s} + o(z^{t-s})\right) \\
            & = (y-x-1)(t-s+1) N^{t-s} + o(N^{t-s}), 
        \end{align*}
        and 
        \begin{align*}
            \frac{(y-s)_{t-s}}{y}-\frac{(x+1-s)_{t-s}}{x+1}
            & = (y-x-1) \psi_{2}'(z_2) \\
            & = (y-x-1) \left((t-s-1) z_{2}^{t-s-2} + o(z^{t-s-2})\right) \\
            & = (y-x-1) (t-s-1) N^{t-s-2} + o(N^{t-s-2}). 
        \end{align*}
        Consequently, we have 
        \begin{align*}
            g(x,y)
            & = s (y-x-1)(t-s+1) N^{t-s} +t (x+1)y (y-x-1) (t-s-1) N^{t-s-2} + o(N^{t-s}) \\
            & = \left(t+s - (t-s)^2\right) (y-x-1) N^{t-s} + o(N^{t-s}),  
        \end{align*}
        which proves Claim~\ref{CLAIM:exact-Mean-Value-1928}. 
    \end{proof}
    It follows from Claim~\ref{CLAIM:exact-Mean-Value-1928} that 
    \begin{align*}
        \tbinom{x+1}{t}\tbinom{y-1}{s} + \tbinom{x+1}{s}\tbinom{y-1}{t} 
          - \tbinom{x}{t}\tbinom{y}{s} - \tbinom{x}{s}\tbinom{y}{t}
        & = \frac{(x)_{s-1}(y-1)_{s-1}}{t!\,s!} \cdot g(x,y) \\
        & = \frac{t+s-(t-s)^2}{s!t!}\bigl(y-x-1\bigr)N^{t+s-2} + o(N^{t+s-2}),  
    \end{align*}
    which completes the proof of Lemma~\ref{LEMMA:exact-Mean-Value}. 
\end{proof}

\begin{lemma}\label{LEMMA:prepare-of-exact}
    Let $F = K_{a_1, \ldots, a_{r}}$ be an almost balanced complete $r$-partite graph with $a_1 \ge \cdots \ge a_r$, and let $\ell \coloneqq a_1 + \cdots + a_{r}$. 
    Let $m\coloneqq m_{r, \ell}$ be the constant given by Lemma~\ref{Lemma:Unique-Maximizer}. 
    Then 
    \begin{align*}
        \binom{\ell}{2} 
        > m \sum_{k \in [r]} \binom{a_k}{2}. 
    \end{align*}
\end{lemma}
\begin{proof}
    Suppose that $a_r = 1$. 
    Then it follows from Fact~\ref{FACT:almost-balanced-sizes}~\ref{FACT:almost-balanced-sizes-1} that $a_1 = \cdots = a_{\ell-r} = 2$ and $a_{\ell-r+1} = \cdots = a_r = 1$.
    Combining this with $m < \frac{\ell(\ell-1)}{2(\ell-r)}$ (by Lemma~\ref{Lemma:m-bounds}~\ref{Lemma:m-upper-bound}), we obtain 
    \begin{align*}
        m \sum_{k \in [r]} \binom{a_k}{2} 
        = m(\ell-r) 
        < \frac{\ell(\ell-1)}{2(\ell-r)}  (\ell-r) 
        = \binom{\ell}{2}.
    \end{align*}
    Now suppose that $a_r \ge 2$ (consequently, $\ell \ge 2r$).
    Observe that 
    \begin{align}\label{equ:almost-balanced-transform}
        r \sum_{k\in [r]} \binom{a_{k}}{2}
        = \frac{1}{2} \Big(r \sum_{k\in [r]} a_{k}^{2} - r \sum_{k\in [r]} a_{k} \Big) 
        & = \frac{1}{2} \Big( (a_1 + \cdots + a_{r})^2 + \sum_{\{i,j\} \in \binom{[r]}{2}} (a_i - a_j)^2 - r\ell \Big) \notag \\
        & = \frac{1}{2} \Big( \ell^2 + \sum_{\{i,j\} \in \binom{[r]}{2}} (a_i - a_j)^2 - r\ell \Big). 
    \end{align}
    For every $\{i,j\} \subseteq [r]$, it follows from the almost balanced assumption that 
    \begin{align*}
        \binom{a_i-a_j}{2} \le \binom{a_1-a_r}{2} \le a_r-1 \le a_j-1, 
    \end{align*}
    which implies that 
    \begin{equation}\label{eq:diff-bound}
        (a_i - a_j)^2 \le a_i + a_j - 2.
    \end{equation}
    
    First, we consider the case $r \le 3$. 
    Since $\ell \ge 2r \ge \frac{\ln(r+1)}{\ln(1+1/r)}$ for $r \in \{2,3\}$, the integer $m_{r,\ell}$ given by Lemma~\ref{Lemma:Unique-Maximizer} satisfies that $m_{r,\ell} = r$.
    Combining this with~\eqref{equ:almost-balanced-transform} and~\eqref{eq:diff-bound}, we obtain 
    \begin{align*}
        m \sum_{k \in [r]}\binom{a_k}{2}
        = r \sum_{k \in [r]}\binom{a_k}{2}
        & = \frac{1}{2} \Big( \ell^2 + \sum_{\{i,j\} \in \binom{[r]}{2}} (a_i - a_j)^2 - r\ell \Big)  \le \frac{1}{2}\Bigl( \ell^2 + \sum_{\{i,j\} \in \binom{[r]}{2}} (a_i+a_j-2) - r\ell \Bigr) \\
        & = \frac{1}{2}\bigl( \ell^2 + (r-1)\ell - r(r-1) - r\ell \bigr) 
        = \frac{1}{2}\bigl( \ell^2 - \ell - r(r-1) \bigr) < \binom{\ell}{2}, 
    \end{align*}
    as desired. So it remains to prove the case when $r \ge 4$. 
    
    \begin{claim}\label{Claim:(a_i-a_j)2-upper-bound}
        We have $\sum_{\{i,j\} \in \binom{[r]}{2}} (a_i-a_j)^2 \le (r-1)\ell - r^2$.
    \end{claim}
    \begin{proof}
        Observe that 
        \begin{align*}
            \sum_{\{i,j\} \in \binom{[r]}{2}} (a_i -a_j)^2
            & = \sum_{\{i,j\} \in \binom{[2,r-1]}{2}} (a_i -a_j)^2 + (a_1 - a_{r})^2 + \sum_{i \in [2,r-1]} \big( (a_1 - a_{i})^2 + (a_i - a_{r})^2 \big) \\
            & \le \sum_{\{i,j\} \in \binom{[2,r-1]}{2}} (a_i -a_j)^2 + (a_1 - a_{r})^2 + \sum_{i \in [2,r-1]} (a_1 - a_{r})^2  \\
            &=  \sum_{\{i,j\} \in \binom{[2,r-1]}{2}} (a_i -a_j)^2 + (r-1) (a_1 - a_{r})^2, 
        \end{align*}
        where the inequality follows from the fact that, for all $x \ge y \ge z \ge 0$, 
        \begin{align*}
            (x-y)^2 + (y-z)^2 
            = \big((x-y) + (x-z)\big)^2 -2(x-y)(y-z)
            \le (x-z)^2.
        \end{align*}
        Repeating this argument for $\sum_{{i,j} \in \binom{[2,r-1]}{2}} (a_i -a_j)^2$, and continuing similarly, then applying~\eqref{eq:diff-bound}, we obtain 
        \begin{align*}
            \sum_{\{i,j\} \in \binom{[r]}{2}} (a_i -a_j)^2 
            & \le (r-1) (a_1 - a_{r})^2 + (r-3) (a_2 - a_{r-1})^2 + \cdots + \big( r - 2 \lfloor \tfrac{r}{2} \rfloor + 1 \big) \big( a_{\lfloor r/2 \rfloor} - a_{r+1-\lfloor r/2 \rfloor} \big)^2 \\
            & = \sum_{i = 1}^{\lfloor r/2 \rfloor} (r-2i+1) (a_i - a_{r+1-i})^2
            \le \sum_{i = 1}^{\lfloor r/2 \rfloor} (r-2i+1) (a_i + a_{r+1-i} - 2).
        \end{align*}
        So it suffices to show that 
        \begin{align*}
            \sum_{i = 1}^{\lfloor r/2 \rfloor} (r-2i+1) (a_i + a_{r+1-i} - 2)
            \le (r-1)\ell - r^2.
        \end{align*}
        Suppose that $r = 2k$ for some $k \ge 2$. 
        Then 
        \begin{align*}
            \sum_{i = 1}^{\lfloor r/2 \rfloor} (r-2i+1) (a_i + a_{r+1-i} - 2)
            & = \sum_{i \in [k]} (r-1) (a_i + a_{r+1-i} - 2) - \sum_{i \in [k]} (2i-2) (a_i + a_{r+1-i} - 2) \\
            & \le (r-1) (\ell - r) - \sum_{i \in [k]} (2i-2)(2+2-2) \\
            & = (r-1) (\ell - r) - \frac{r(r-2)}{2}  
            = (r-1)\ell - r^2 - \frac{r(r-4)}{2}
            \le (r-1)\ell - r^2, 
        \end{align*}
        as desired. 
        
        Now suppose that $r = 2k-1$ for some $k \ge 3$. 
        Then 
        \begin{align*}
            & \sum_{i = 1}^{\lfloor r/2 \rfloor} (r-2i+1) (a_i + a_{r+1-i} - 2) \\
            & \quad =  \sum_{i \in [k-1]} (r-1) (a_i + a_{r+1-i} - 2) - \sum_{i \in [k-1]} (2i-2) (a_i + a_{r+1-i} - 2) \\
            & \quad =  \sum_{i = 1}^{r} (r-1) (a_i -1) - (r-1)(a_k-1) - \sum_{i \in [k-1]} (2i-2) (a_i + a_{r+1-i} - 2) \\
            & \quad \le  (r-1) (\ell - r) - (r-1)(2-1) - \sum_{i \in [k-1]} (2i-2) (2+2 - 2) \\
            & \quad =  (r-1) (\ell - r) - (r-1)- \frac{1}{2} \left(r^2-4 r+3\right)  
             = (r-1)\ell - r^2 - \frac{r(r-4)+1}{2}
            \le (r-1)\ell - r^2, 
        \end{align*}
        also as desired. 
        This completes the proof of Claim~\ref{Claim:(a_i-a_j)2-upper-bound}.
    \end{proof}

    It follows from~\eqref{equ:almost-balanced-transform}, Claim~\ref{Claim:(a_i-a_j)2-upper-bound}, and the inequality $m < \frac{r \ell^2}{\ell^2 - r^2}$ (by Lemma~\ref{Lemma:m-bounds}~\ref{Lemma:m-upper-bound}) that 
    \begin{align*}
        r\sum_{k \in [r]}\tbinom{a_k}{2}
         = \frac{m}{r} \cdot r\sum_{k \in [r]}\tbinom{a_k}{2} 
        & = \frac{m}{r} \cdot \frac{1}{2} \Big(\ell^2 + \sum_{\{i,j\} \in \binom{[r]}{2}} (a_i -a_j)^2 - r \ell \Big) \\
        & \le \frac{\ell^2}{\ell^2 - r^2} \Big(\ell^2 + (r-1)\ell - r^2 - r \ell \Big) 
        = \binom{\ell}{2} - \frac{r^2 \ell}{2(\ell^2 - r^2)}
        < \binom{\ell}{2}. 
    \end{align*}
    This completes the proof of Lemma~\ref{LEMMA:prepare-of-exact}.
\end{proof}

We are now ready to prove Theorem~\ref{THM:almost-balanced-exact}.
\begin{proof}[\bf Proof of Theorem~\ref{THM:almost-balanced-exact}]
    Let $F = K_{a_1, \ldots, a_{r}}$ be an almost balanced complete $r$-partite graph on $\ell$ vertices, that is, $\ell = a_1 + \cdots + a_r$.
    Let $m = m_{r,\ell}$ be the integer given by Lemma~\ref{Lemma:Unique-Maximizer}.
    Let $n$ be a sufficiently large integer, and let $N \coloneqq n/m$. 
    Let $G$ be an $n$-vertex graph satisfying $I(F, G) = I(F, n)$. 
    
    It follows from Theorems~\ref{THM:almost-balanced-asymptotic} and~\ref{THM:almost-balanced-perfect-stability} that $G$ is complete $m$-partite with each part of size $(1+o(1))N$. 
    Write the part sizes of $G$ as $n_1 \geq n_2 \geq \dots \geq n_m$, noting that $\sum_{i=1}^m n_i = n$ and $|n_i - N| = o(n)$ for every $i \in [m]$. 
    Let $y \coloneqq n_1$ and $x \coloneqq n_m$. 
    Suppose to the contrary that $G$ is not a Tur\'{a}n graph, that is, $y \ge x+2$.  

    For convenience, for every graph $H$, let $\hat{I}(F,H) \coloneqq \mathrm{sym}(a_1,\dots,a_r) \cdot I(F,H)$. 
    
    \begin{claim}\label{CLAIM:exact-n-large-shift}
        The complete $m$-partite graph $G_{\ast} \cong K_{n_1-1, n_2, \dots, n_m+1}$ satisfies 
        \begin{align*}
            \hat{I}(F,G_{\ast}) - \hat{I}(F,G) 
            > 0. 
        \end{align*}
    \end{claim}
    \begin{proof}   
        Recall that $m=2$ when $r=2$ by Lemma~\ref{Lemma:Unique-Maximizer}. 
        For each $i \in [r]$, let 
        \begin{align*}
            f_{i}(x,y)
            & \coloneqq \tbinom{x+1}{a_i} + \tbinom{y-1}{a_i} - \tbinom{x}{a_i} - \tbinom{y}{a_i},  \\[.4em]
            A_i 
            & \coloneqq 
            \begin{cases}
                0, & \quad\text{if}\quad r=2, \\[.3em]
                \sum_{(s_k)_{k \in [r]\setminus \{i\}}} \prod_{k \in [r]\setminus \{i\}} \tbinom{n_{s_k}}{a_k}, & \quad\text{if}\quad r \ge 3, 
            \end{cases}
            \quad = \left( \tfrac{(m-2)_{r-1} \cdot a_i!}{a_1! \cdots a_{r}!} + o(1)\right) N^{\ell-a_i}.
        \end{align*}
        where the tuple $(s_k)_{k \in [r]\setminus \{i\}}$ ranges over $\big([2,m-1]\big)_{r-1}$ in the summation.  
        
        For each $\{i, j\} \subseteq [r]$, let 
        \begin{align*}
            g_{i,j}(x,y)
            & \coloneqq \tbinom{x+1}{a_i}\tbinom{y-1}{a_j} + \tbinom{x+1}{a_j}\tbinom{y-1}{a_i} - \tbinom{x}{a_i}\tbinom{y}{a_j} - \tbinom{x}{a_j}\tbinom{y}{a_i},  \\[.4em]
            B_{i,j}
            & \coloneqq 
            \begin{cases}
                1, & \quad\text{if}\quad r=2, \\[.3em] 
                \sum_{(s_k)_{k \in [r]\setminus \{i,j\}}} \prod_{k \in [r]\setminus \{i,j\}} \tbinom{n_{s_k}}{a_k}, & \quad\text{if}\quad r \ge 3, 
            \end{cases}
            \quad = \left( \tfrac{(m-2)_{r-2}\cdot a_i! a_j!}{a_1! \cdots a_{r}!} + o(1) \right) N^{\ell-a_i-a_j}, 
        \end{align*}
        where the tuple $(s_k)_{k \in [r]\setminus \{i,j\}}$ ranges over $\big([2,m-1]\big)_{r-2}$ in the summation.

        Since $\sc(G) = 0$, applying~\eqref{equ:IFG-formula} and simplifying, we obtain
        \begin{align}\label{equ:exact-shift-diff}
            \hat{I}(F,G_{\ast}) - \hat{I}(F,G) 
            & = \sum_{i \in [r]} A_i \cdot f_{i}(x,y) + \sum_{\{i,j\} \in \binom{[r]}{2}} B_{i,j} \cdot g_{i,j}(x,y). 
        \end{align}
        It follows from~\eqref{equ:LEMMA:exact-Mean-Value-a} that 
        \begin{align}\label{equ:exact-shift-diff-A}
            A_i \cdot f_{i}(x,y)
            & = - \frac{a_i (a_i-1)}{a_i!} (y-x-1+o(1)) N^{a_i - 2} \cdot \left( \frac{(m-2)_{r-1} \cdot a_i!}{a_1! \cdots a_{r}!} + o(1)\right) N^{\ell-a_i} \notag \\
            & = a_i (a_i-1) (y-x-1) \frac{(m-2)_{r-1}}{a_1! \cdots a_{r}!} N^{\ell-2} + o(N^{\ell-2}). 
        \end{align}
        It follows from~\eqref{equ:LEMMA:exact-Mean-Value-b} that 
        \begin{align}\label{equ:exact-shift-diff-B}
            B_{i,j} \cdot g_{i,j}(x,y) 
            & = \frac{  a_i + a_j - (a_i - a_j)^2 }{a_i!a_j!} \big( y-x-1 +o(1) \big) N^{a_i+a_j-2} \cdot \left( \frac{(m-2)_{r-2}\cdot a_i! a_j!}{a_1! \cdots a_{r}!} + o(1) \right) N^{\ell-a_i-a_j} \notag \\
            & = \left(a_i + a_j - (a_i - a_j)^2\right) (y-x-1) \frac{(m-2)_{r-2}}{a_1! \cdots a_{r}!} N^{\ell-2} + o(N^{\ell-2}). 
        \end{align}
        Let 
        \begin{align*}
            \Psi
            \coloneqq - \sum_{i\in [r]} (m-r)a_i(a_i-1) + \sum_{\{i,j\}\in \binom{[r]}{2}} \left(a_i + a_j - (a_i - a_j)^2\right). 
        \end{align*}
        Then 
        \begin{align*}
            \Psi
            & = - \sum_{i\in [r]} (m-r)(a_i^2 - a_i) + \sum_{\{i,j\}\in \binom{[r]}{2}} \left(a_i + a_j - a_i^2 - a_j^2 + 2a_i a_j\right) \\
            & = \sum_{i\in [r]} \left( -(m-r)(a_i^2 - a_i) - (r-1) (a_i^2 - a_i) \right) + \sum_{i\in [r]} a_i \sum_{j \in [r]\setminus \{i\}} a_j \\
            & = - \sum_{i\in [r]} (m-1)(a_i^2 - a_i) + \sum_{i\in [r]} a_i (\ell-a_i) \\
            & = - \sum_{i\in [r]} m (a_i^2 - a_i) + \sum_{i\in [r]} (\ell-1) a_i
            = 2\Big( \tbinom{\ell}{2} - m \sum_{i\in [r]} \tbinom{a_i}{2} \Big). 
        \end{align*}
        Combining this with~\eqref{equ:exact-shift-diff},~\eqref{equ:exact-shift-diff-A}, and~\eqref{equ:exact-shift-diff-B}, we obtain 
        \begin{align*}
            \hat{I}(F,G_{\ast}) - \hat{I}(F,G)
            & = \Psi \cdot (y-x-1) \frac{(m-2)_{r-2}}{a_1! \cdots a_{r}!} N^{\ell-2} + o(N^{\ell-2}) \\
            & = 2\Big( \tbinom{\ell}{2} - m \sum_{i\in [r]} \tbinom{a_i}{2} \Big) \cdot (y-x-1) \frac{(m-2)_{r-2}}{a_1! \cdots a_{r}!} N^{\ell-2} + o(N^{\ell-2}). 
        \end{align*}
        By Lemma~\ref{LEMMA:prepare-of-exact}, the factor $\binom{\ell}{2} - m\sum_{i} \binom{a_i}{2}$ is positive and hence at least one.
        It follows that $\hat{I}(F,G_{\ast}) - \hat{I}(F,G) > 0$, which establishes Claim~\ref{CLAIM:exact-n-large-shift}.
    \end{proof}
    
    It follows from Claim~\ref{CLAIM:exact-n-large-shift} that $G$ is not extremal, contradicting the assumption that $I(F,G) = I(F,n)$. 
    This completes the proof of Theorem~\ref{THM:almost-balanced-exact}.
\end{proof}

For the balanced case $K_r(t)=T_r(rt)$, the requirement that $n$ be sufficiently large can be dropped, though the extremal graph may no longer be unique.

\begin{theorem}\label{THM:Kr(t)-exact}
    Let $r, t \geq 2$ be integers. The following holds for every $n \ge rt$. 
    Suppose that $G$ is an $n$-vertex complete multipartite graph satisfying $I(K_r(t),G) = I(K_r(t), n)$. Then $G \cong T_k(n)$ for some integer $k \ge r$. 
    In particular, for every integer $n \ge rt$, there exists an integer $k$ such that $I(K_r(t),T_k(n)) = I(K_r(t), n)$. 
\end{theorem}
\begin{proof}
    Fix integers $r \ge 2$ and $t \ge 2$, and let $F = K_r(t)$.
    Fix $n \ge rt$. 
    By Lemma~\ref{Lemma:Sidorenko-Extremal-Partite}, there exists a complete multipartite graph $G$ such that $I(F,G) = I(F,n)$.
    Write $G = K_{a_1,\ldots,a_k}$ for some integer $k$, and assume that $a_1 \ge \cdots \ge a_k$.
    Clearly, we must have $k \ge r$ and $a_k \ge t$, since otherwise $I(F,G) = 0$.

    Suppose to the contrary that $a_1 \ge a_k + 2$. 
    Let $G'$ be the complete $k$-partite graph with part sizes $a_1-1, a_2, \cdots, a_{k-1}, a_k+1$, and let $G''$ be the complete $(k-1)$-partite graph with part sizes $a_1 + a_k, a_2, \cdots, a_{k-1}$.

    \begin{claim}\label{Claim:Adjustment-Method-Discrete}
        We have $I(F,G) < \max\{ I(F,G'),~I(F,G'') \}$.
    \end{claim}
    \begin{proof}
        Let $x \coloneqq a_k$ and $y \coloneqq a_1$, noting that $y-1 > x \geq t$. 
        Define
        \begin{align*}
            A 
            & \coloneqq 
            \begin{cases}
                0, & \quad\text{if}\quad r=2, \\[.4em]
                r \cdot \sum_{(i_1,\dots,i_{r-1}) \in ([2,k-1])_{r-1}} 
                \tbinom{a_{i_1}}{t} \cdots \tbinom{a_{i_{r-1}}}{t}, & \quad\text{if}\quad r \ge 3, 
            \end{cases} \\[.4em]
            B 
            & \coloneqq 
            \begin{cases}
                2, & \quad\text{if}\quad r=2, \\[.4em]
                r(r-1) \cdot \sum_{(i_1,\dots,i_{r-2}) \in ([2,k-1])_{r-2}} 
                \tbinom{a_{i_1}}{t} \cdots \tbinom{a_{i_{r-2}}}{t}, & \quad\text{if}\quad r \ge 3. 
            \end{cases}
        \end{align*}
        %
        Since $\min_{i \in [k]}a_i \ge t$ and $k \ge r$, it follows directly from the definition that $B > 0$. 

        Suppose for a contradiction that both $I(F,G) \ge I(F,G')$ and $I(F,G) \ge I(F,G'')$ hold.
        Applying~\eqref{equ:IFG-formula} and then simplifying, we obtain 
        \begin{align}
            0 &\le r! \bigl( I(F,G) - I(F,G') \bigr) 
               = \alpha_1 A - \beta_1 B,
                \label{eq:exact-adjust-1} \\
            0 &\le r! \bigl( I(F,G) - I(F,G'') \bigr) 
               = \beta_2 B 
                 - \alpha_2 A,
                \label{eq:exact-adjust-2}
        \end{align}
        where 
        \begin{align*}
            \alpha_1
            \coloneqq \tbinom{x}{t} + \tbinom{y}{t} - \tbinom{x+1}{t} - \tbinom{y-1}{t}, \quad
            \alpha_2
            \coloneqq \tbinom{x+y}{t} - \tbinom{x}{t} - \tbinom{y}{t}, \quad
            \beta_1
            \coloneqq \tbinom{x+1}{t} \tbinom{y-1}{t} - \tbinom{x}{t} \tbinom{y}{t}, \quad
            \beta_2
            \coloneqq \tbinom{x}{t} \tbinom{y}{t}.
        \end{align*}
        
        Since $y > x \geq t \geq 2$, we have
        \begin{align*}
            \alpha_1 = \tbinom{y-1}{t-1} - \tbinom{x}{t-1} > 0, \quad
            \alpha_2 = \sum_{i \in [t-1]} \tbinom{x}{i} \tbinom{y}{t-i} > 0,\quad
            \beta_1 = \tfrac{y-x-1}{t} \tbinom{x}{t-1} \tbinom{y-1}{t-1} > 0, \quad
            \beta_2 > 0.
        \end{align*}
        Combining this with~\eqref{eq:exact-adjust-1},~\eqref{eq:exact-adjust-2}, and $B>0$, we obtain 
        \begin{align*}
            \frac{\beta_1}{\alpha_1} \leq \frac{A}{B} \leq \frac{\beta_2}{\alpha_2},
        \end{align*}
        which after rearranging yields 
        \begin{align*}
            0  \leq \beta_2 \alpha_1 - \beta_1 \alpha_2
            & =  \tbinom{x}{t} \tbinom{y}{t} \Bigl( \tbinom{x}{t} + \tbinom{y}{t} - \tbinom{x+1}{t} - \tbinom{y-1}{t} \Bigr)
            - \Bigl( \tbinom{x+1}{t} \tbinom{y-1}{t} - \tbinom{x}{t} \tbinom{y}{t} \Bigr) 
            \Bigl( \tbinom{x+y}{t} - \tbinom{x}{t} - \tbinom{y}{t} \Bigr) \\[.3em]
            & = \tbinom{y}{t}\tbinom{y-1}{t}\Bigl( \tbinom{x+1}{t}-\tbinom{x}{t} \Bigr)
            -\tbinom{x+1}{t}\tbinom{x}{t}\Bigl(\tbinom{y}{t}- \tbinom{y-1}{t} \Bigr) 
            -  \tbinom{x+y}{t}\Bigl( \tbinom{x+1}{t} \tbinom{y-1}{t} - \tbinom{x}{t} \tbinom{y}{t} \Bigr) \\[.3em]
            & = \tbinom{y}{t}\tbinom{y-1}{t}\tbinom{x}{t-1}-\tbinom{x+1}{t}\tbinom{x}{t} \tbinom{y-1}{t-1}
            -\tfrac{y-x-1}{t} \tbinom{x}{t-1} \tbinom{y-1}{t-1}  \tbinom{x+y}{t} \\[.3em]
            & = \tfrac{1}{t} \tbinom{x}{t-1} \tbinom{y-1}{t-1} 
            \Bigl( y\tbinom{y-1}{t} - (x+1)\tbinom{x}{t} - (y-x-1)\tbinom{x+y}{t} \Bigr).
        \end{align*}
        It follows that 
        \begin{align}\label{equ:Balanced-exact-a}
            y\tbinom{y-1}{t} - (x+1)\tbinom{x}{t} - (y-x-1)\tbinom{x+y}{t}
            \ge 0.
        \end{align}
        Let $\phi \colon [t, \infty) \to \mathbb{R}$ be the function defined by $\phi(z) = \binom{z}{t}$ for every $z\in [t, \infty)$. 
        Since $\phi$ is strictly convex on $[t,\infty)$ for $t \ge 2$, it follows from Jensen's inequality that 
        \begin{align*}
            \tfrac{x+1}{y} \tbinom{x}{t} + \tfrac{y-x-1}{y} \tbinom{x+y}{t} 
            > \tbinom{\frac{x+1}{y} \cdot x + \frac{y-x-1}{y} \cdot (x+y)}{t}
            = \tbinom{y-1}{t}. 
        \end{align*}
        However, this is a contradiction to~\eqref{equ:Balanced-exact-a}. 
        This completes the proof of Claim~\ref{Claim:Adjustment-Method-Discrete}.
    \end{proof}

    It follows from Claim~\ref{Claim:Adjustment-Method-Discrete} that $G$ is not extremal, contradicting the assumption that $I(F,G) = I(F,n)$. 
    This completes the proof of Theorem~\ref{THM:Kr(t)-exact}.
\end{proof}
\section{Concluding remarks}\label{SEC:remarks}
%
%
Let $F$ be a complete $r$-graph on $\ell \ge r+1$ vertices. Observe that if a maximizer $\bm{x}=(x_1,x_2,\ldots)\in \OPT(F)$ satisfies $x_0=0$ and $x_i=x_j$ for all ${i,j}\subseteq \supp(\bm{x})$, then necessarily $\bm{x}= \bm{m}_{r,\ell}$, where $m_{r,\ell}$ is the integer given by Lemma~\ref{Lemma:Unique-Maximizer}.

A natural next step beyond Theorem~\ref{THM:almost-balanced-exact} is to understand to what extent the “almost balanced” hypothesis can be relaxed.

\begin{problem}\label{Prob:asymptotic-equal}
    For all integers $\ell > r \ge 3$, characterize the complete $r$-partite graphs $F$ on $\ell$ vertices for which $\OPT(F)=\{\bm{m}_{r,\ell}\}$.
\end{problem}

We remark that a necessary condition for a complete $r$-partite graph $F = K_{a_1, \ldots, a_{r}}$ on $\ell$ vertices satisfying $\OPT(F)=\{\bm{m}_{r,\ell}\}$ is given by the following inequality, which is slightly weaker than that in Lemma~\ref{LEMMA:prepare-of-exact}:
\begin{align}\label{eq:balanced-OPT-necessary}
    \binom{\ell}{2} \ge m_{r,\ell} \sum_{k \in [r]} \binom{a_k}{2}.
\end{align}
Indeed, let $m=m_{r,\ell}$. 
If~\eqref{eq:balanced-OPT-necessary} fails, consider the perturbation (for a sufficiently small constant $\delta > 0$) 
\begin{align*}
    \bm{x}_\delta
    \coloneqq \big(\tfrac{1+\delta}{m},\underbrace{\tfrac{1}{m},\ldots,\tfrac{1}{m}}_{m-2},\tfrac{1-\delta}{m},0,\ldots \big).
\end{align*}
By~\eqref{equ:pF-x0-zero}, we have
\begin{align*}
    p_F(\bm{x}_\delta) 
    = \frac{\kappa_F(m-2)_{r-2}}{m^{\ell}} 
    \bigl((m-r)(m-r-1)+(m-r)S(\delta)+ T(\delta)\bigr),
\end{align*}
where
\begin{align*}
    S(\delta) 
    \coloneqq \sum_{i\in[r]} \bigl((1+\delta)^{a_i} + (1-\delta)^{a_i}\bigr) 
    \quad \text{and} \quad 
    T(\delta) 
    \coloneqq \sum_{\{i,j\}\subseteq[r]} (1+\delta)^{a_i}(1-\delta)^{a_j}.
\end{align*}
Writing $q(\delta) \coloneqq (m-r)S(\delta)+ T(\delta)$, a direct calculation yields
\begin{align*}
    q'(0) = 0
    \quad\text{and}\quad 
    q''(0) = 2\Bigl(m\sum_{i=1}^r \tbinom{a_i}{2} - \tbinom{\ell}{2} \Bigr) > 0. 
\end{align*}
It follows that $\bm m$ cannot be a maximizer, as claimed.

It should be pointed out that~\eqref{eq:balanced-OPT-necessary} is not a sufficient condition for $\OPT(F)=\{\bm{m}_{r,\ell}\}$ to hold (even when replaced by a strict inequality).
In forthcoming work~\cite{GLMZ26}, we will determine $i(F)$ and establish the perfect stability for all complete $3$-partite graphs $F$. 
In particular, our result will show that $F = K_{12,7,7}$ satisfies~\eqref{eq:balanced-OPT-necessary} with strict inequality. However, the unique maximizer in $\OPT(F)$ is  $\left(\alpha, \frac{1-\alpha}{2}, \frac{1-\alpha}{2}, 0, \ldots \right)$, where $\alpha \approx 0.396884$ is the largest real root of 
\begin{align*}
    130x^5 + 25x^4 - 90x^3 + 80x^2 - 40x + 7 = 0. 
\end{align*}

On the other hand, the sufficient condition~\eqref{equ:def-almost-balance} is not necessary in general. For example, $F = K_{4,8,8}$ is not almost balanced (since $\binom{8-4}{2} = 6 > 4$), yet our result will show that the unique maximizer for $F$ is $(1/3, 1/3, 1/3, 0, \ldots)$. 

Thus, the answer to Problem~\ref{Prob:asymptotic-equal}, in terms of the algebraic relations among $a_1, \ldots, a_r$, lies strictly between the hypersurfaces defined by~\eqref{equ:def-almost-balance} and~\eqref{eq:balanced-OPT-necessary}, and appears to be quite complicated.

\section*{Acknowledgments}
X.L. was supported by the Excellent Young Talents Program (Overseas) of the National Natural Science Foundation of China. 
J.M. was supported by National Key Research and Development Program of China 2023YFA1010201 and National Natural Science Foundation of China grant 12125106.
T.Z. was supported by Innovation Program for Quantum Science and Technology 2021ZD0302902.
X.L. would like to thank Oleg Pikhurko for clarifying the necessity of considering the case $x_0 > 0$ in the proof of Theorem~\ref{THM:almost-balanced-asymptotic}.
\bibliographystyle{abbrv}
\bibliography{Inducibility}

\begin{thebibliography}{10}

\bibitem{AHKT20}
N.~Alon, D.~Hefetz, M.~Krivelevich, and M.~Tyomkyn.
\newblock Edge-statistics on large graphs.
\newblock {\em Combin. Probab. Comput.}, 29(2):163--189, 2020.

\bibitem{Balogh16}
J.~Balogh, P.~Hu, B.~Lidick{\`y}, and F.~Pfender.
\newblock Maximum density of induced 5-cycle is achieved by an iterated blow-up of 5-cycle.
\newblock {\em European J. Combin.}, 52:47--58, 2016.

\bibitem{Basit25}
A.~Basit, B.~Granet, D.~Horsley, A.~K{\"u}ndgen, and K.~Staden.
\newblock The semi-inducibility problem.
\newblock {\em arXiv preprint arXiv:2501.09842}, 2025.

\bibitem{BGLLPS26}
L.~Bodn{\'a}r, J.~Gao, J.~Le{\'o}n, X.~Liu, O.~Pikhurko, and S.~Sun.
\newblock The inducibility of 6-vertex graphs.
\newblock Manuscript in preparation.

\bibitem{Bollobas95}
B.~Bollob{\'a}s, Y.~Egawa, A.~Harris, and G.~Jin.
\newblock The maximal number of induced {$r$}-partite subgraphs.
\newblock {\em Graphs Combin.}, 11(1):1--19, 1995.

\bibitem{BollobasNara86}
B.~Bollob{\'a}s, C.~Nara, and S.-i. Tachibana.
\newblock The maximal number of induced complete bipartite graphs.
\newblock {\em Discrete Math.}, 62(3):271--275, 1986.

\bibitem{Brown94}
J.~I. Brown and A.~Sidorenko.
\newblock The inducibility of complete bipartite graphs.
\newblock {\em J. Graph Theory}, 18(6):629--645, 1994.

\bibitem{Erd84}
P.~Erd\H{o}s.
\newblock On some problems in graph theory, combinatorial analysis and combinatorial number theory.
\newblock In {\em Graph theory and combinatorics ({C}ambridge, 1983)}, pages 1--17. Academic Press, London, 1984.

\bibitem{Even15note}
C.~Even-Zohar and N.~Linial.
\newblock A note on the inducibility of 4-vertex graphs.
\newblock {\em Graphs Combin.}, 31(5):1367--1380, 2015.

\bibitem{Exo86}
G.~Exoo.
\newblock Dense packings of induced subgraphs.
\newblock {\em Ars Combin.}, 22:5--10, 1986.

\bibitem{FHL17}
J.~Fox, H.~Huang, and C.~Lee.
\newblock A solution to the inducibility problem for almost all graphs.
\newblock Preprint, 2017.

\bibitem{FS20}
J.~Fox and L.~Sauermann.
\newblock A completion of the proof of the edge-statistics conjecture.
\newblock {\em Adv. Comb.}, pages Paper No. 4, 52, 2020.

\bibitem{GLMZ26}
J.~Gao, X.~Liu, J.~Ma, and T.~Zhu.
\newblock The inducibility of complete tripartite graphs.
\newblock Manuscript in preparation.

\bibitem{Gre12}
A.~Grzesik.
\newblock On the maximum number of five-cycles in a triangle-free graph.
\newblock {\em J. Combin. Theory Ser. B}, 102(5):1061--1066, 2012.

\bibitem{Hatami14}
H.~Hatami, J.~Hirst, and S.~Norine.
\newblock The inducibility of blow-up graphs.
\newblock {\em J. Combin. Theory Ser. B}, 109:196--212, 2014.

\bibitem{HHKNR13}
H.~Hatami, J.~Hladk\'{y}, D.~Kr\'{a}\v{l}, S.~Norine, and A.~Razborov.
\newblock On the number of pentagons in triangle-free graphs.
\newblock {\em J. Combin. Theory Ser. A}, 120(3):722--732, 2013.

\bibitem{HT18}
D.~Hefetz and M.~Tyomkyn.
\newblock On the inducibility of cycles.
\newblock {\em J. Combin. Theory Ser. B}, 133:243--258, 2018.

\bibitem{Hir14}
J.~Hirst.
\newblock The inducibility of graphs on four vertices.
\newblock {\em J. Graph Theory}, 75(3):231--243, 2014.

\bibitem{KNV19}
D.~Kr{\'a}{l'}, S.~Norin, and J.~Volec.
\newblock A bound on the inducibility of cycles.
\newblock {\em J. Combin. Theory Ser. A}, 161:359--363, 2019.

\bibitem{KST19}
M.~Kwan, B.~Sudakov, and T.~Tran.
\newblock Anticoncentration for subgraph statistics.
\newblock {\em J. Lond. Math. Soc. (2)}, 99(3):757--777, 2019.

\bibitem{LP18}
B.~Lidick\'{y} and F.~Pfender.
\newblock Pentagons in triangle-free graphs.
\newblock {\em European J. Combin.}, 74:85--89, 2018.

\bibitem{LiuPikhurko23}
H.~Liu, O.~Pikhurko, M.~Sharifzadeh, and K.~Staden.
\newblock Stability from graph symmetrisation arguments with applications to inducibility.
\newblock {\em J. Lond. Math. Soc. (2)}, 108(3):1121--1162, 2023.

\bibitem{LiuMubayi23}
X.~Liu, D.~Mubayi, and C.~Reiher.
\newblock The feasible region of induced graphs.
\newblock {\em J. Combin. Theory Ser. B}, 158:105--135, 2023.

\bibitem{MMNT19}
A.~Martinsson, F.~Mousset, A.~Noever, and M.~Truji\'c.
\newblock The edge-statistics conjecture for {$\ell\ll k^{6/5}$}.
\newblock {\em Israel J. Math.}, 234(2):677--690, 2019.

\bibitem{NY17tri}
S.~Norin and L.~Yepremyan.
\newblock Tur{\'a}n number of generalized triangles.
\newblock {\em J. Combin. Theory Ser. A}, 146:312--343, 2017.

\bibitem{Sliacan19}
O.~Pikhurko, J.~Slia{\v{c}}an, and K.~Tyros.
\newblock Strong forms of stability from flag algebra calculations.
\newblock {\em J. Combin. Theory Ser. B}, 135:129--178, 2019.

\bibitem{Pippenger75}
N.~Pippenger and M.~C. Golumbic.
\newblock The inducibility of graphs.
\newblock {\em J. Combin. Theory Ser. B}, 19(3):189--203, 1975.

\bibitem{Raz07}
A.~A. Razborov.
\newblock Flag algebras.
\newblock {\em J. Symbolic Logic}, 72(4):1239--1282, 2007.

\bibitem{Sim68}
M.~Simonovits.
\newblock A method for solving extremal problems in graph theory, stability problems.
\newblock In {\em Theory of {G}raphs ({P}roc. {C}olloq., {T}ihany, 1966)}, pages 279--319. Academic Press, New York-London, 1968.

\bibitem{Uel24}
R.~Ueltzen.
\newblock Characterizing graphs with high inducibility.
\newblock {\em arXiv preprint arXiv:2411.17362}, 2024.

\bibitem{Yuster19}
R.~Yuster.
\newblock On the exact maximum induced density of almost all graphs and their inducibility.
\newblock {\em J. Combin. Theory Ser. B}, 136:81--109, 2019.

\bibitem{Yuster26}
R.~Yuster.
\newblock Inducibility in {$H$}-free graphs and inducibility of {T}ur{\'a}n graphs.
\newblock {\em J. Combin. Theory Ser. B}, 178:1--26, 2026.

\bibitem{Zyk49}
A.~A. Zykov.
\newblock On some properties of linear complexes.
\newblock {\em Mat. Sbornik N.S.}, 24/66:163--188, 1949.

\end{thebibliography}

\end{document}